\documentclass[11pt]{amsart}
\usepackage{graphicx}
\usepackage[usenames]{color}
\usepackage{amssymb,amscd}
\usepackage{comment}

\usepackage{subfiles}	
\usepackage{hyperref}
\hypersetup{
    colorlinks,
    citecolor=black,
    filecolor=black,
    linkcolor=black,
    urlcolor=black
}

\newcommand{\bea}{\begin{eqnarray*}}
\newcommand{\eea}{\end{eqnarray*}}

\newcommand{\bm}{\begin{pmatrix}}
\newcommand{\fm}{\end{pmatrix}}

\newcommand\Z{\mathbb Z}

\newcommand\R{\mathbb R}
\newcommand{\ra}{\rightarrow}

\newcommand{\g}{\gamma}
\newcommand{\G}{\Gamma}

\newtheorem{theorem}{Theorem}[section]

\newtheorem{Prop}[theorem]{Proposition}
\newtheorem{lemma}[theorem]{Lemma}
\newtheorem{remark}[theorem]{Remark}

\newtheorem{corollary}[theorem]{Corollary}
\newtheorem{definition}[theorem]{Definition}

\title{A fibered Tukia theorem for nilpotent Lie groups}
\author{Tullia Dymarz}
\address{Department of Mathematics,
University of Wisconsin-Madison, 480 Lincoln Drive,  Madison, WI 53706} \email{dymarz@math.wisc.edu}

\author{David Fisher}
\address{Department of Mathematics,
 Rice  University,  6100 Main Street
Houston, TX 77005}
   \email{davidfisher@rice.edu}

\author{Xiangdong Xie}
\address{Department of Mathematics and Statistics,
Bowling Green State University,
Bowling Green, OH 43403} \email{xiex@bgsu.edu}
\thanks{}

\begin{document}
\maketitle

\begin{abstract}
We establish a Tukia-type  theorem for uniform quasiconformal groups of  a Carnot group. More generally we establish a fiber bundle version (or foliated version)  of Tukia theorem for
uniform quasiconformal groups of  a nilpotent Lie group   whose Lie algebra admits a diagonalizable derivation with positive eigenvalues. These results have applications to quasi-isometric rigidity of solvable groups \cite{DFX}.

\end{abstract}

\section{Introduction}

In this paper we study uniform quasiconformal groups of simply connected nilpotent Lie groups.  The nilpotent Lie groups considered in this paper
are those  whose Lie algebras admit a diagonalizable derivation with positive eigenvalues.   We start  with the special case of Carnot groups.

  Let $N$ be  a Carnot group
    equipped with a left invariant Carnot-Caratheodory metric $d_{CC}$.
    Let $\hat N=N\cup \{\infty\}$ be the one-point compactification of $N$.
     A homeomorphism  $f: \hat N \ra \hat N$ is $K$-quasiconformal for some $K\ge 1$ if
      $f: (N\backslash\{f^{-1}(\infty)\}, d_{CC})  \ra (N\backslash \{f(\infty)\}, d_{CC})$      is $K$-quasiconformal.
        A $1$-quasiconformal map is also called conformal.
          A group $G$ of homeomorphisms of $\hat N$ is  \emph{uniformly quasiconformal}
              if there is some $K\ge 1$ such that every $g\in G$ is $K$-quasiconformal.
    If $G'$ is a conformal group of $\hat N$ and $f$ is a self quasiconformal map of $\hat N$, then $fG'f^{-1}$ is a uniform quasiconformal group of $\hat N$. A natural question is when
     a uniform quasiconformal group of $\hat N$ arises this way.

    Let $G$ be a group of homeomorphisms of $\hat N$.  Then $G$ also acts  diagonally  on the space of distinct triples
     $$T(\hat N)=\left\{(x,y,z):   x,y, z\in \hat N, x\not=y, y\not=z, z\not=x\right\}$$
    of $\hat N$.   Our first result is the following.

   \begin{theorem}\label{tukia}
    Let $N$ be  a Carnot group.     There is    a left invariant Carnot-Caratheodory metric $d_0$ on $N$ with the following property.
    Let $G$ be a 
      uniform  quasiconformal group of  $\hat N$.
     If the action of $G$ on the space of distinct triples of $\hat N$ is co-compact, then
       there is some quasiconformal map $f: \hat N \ra \hat N$
         such that $fGf^{-1}$    consists of conformal  maps  with respect to $d_0$.
   \end{theorem}

         The metric $d_0$ has the largest conformal group  in the sense that  the conformal group of any  left invariant Carnot-Caratheodory metric is conjugated into the conformal group of $d_0$, see  Definition \ref{max-def}     and  Lemma \ref{maximal}.  In general it is not possible to conjugate a uniform quasiconformal group into the conformal group of an arbitrary  left invariant Carnot-Caratheodory metric, see   Section \ref{example}  for an example.

    Theorem \ref{tukia} was first  established by Tukia  \cite{T86} for $N=\R^n$ ($n\ge 2$) and  was  later generalized  by Chow \cite{Ch96}
        to the case
     when $N$ is an  Heisenberg group.    Before Tukia's result,
       D.  Sullivan \cite{S78}  proved that, when $n=2$   every uniform quasiconformal  group (without the assumption of cocompactness of the induced action on the space of distinct triples)  is
        quasiconformally    conjugate to a conformal
group.

     Tukia's theorem   has
       applications  to  rigidity of quasi-actions and
         quasi-isometric rigidity of finitely generated groups.
           So does Theorem \ref{tukia}.  A \emph{quasi-action} of a group $\Gamma$ on a metric space $X$ is an assignment $ \gamma \mapsto G_\gamma$ where $G_\gamma$ is a self quasi-isometry of $X$ such that
\begin{enumerate}
\item  $G_\gamma$ is an $(L,A)$ quasi-isometry where $L$ and $A$ are uniform over all $\gamma \in \Gamma$.
\item  $G_{\gamma\eta}$ and  $G_\gamma \circ G_\eta$ are  bounded distance apart in the sup norm, uniformly over all $\gamma, \eta \in \Gamma$.
\item $G_{Id}$ is bounded distance from the identity map on $X$.
\end{enumerate}
A quasi-action is \emph{cobounded} if there is a bounded set $S \subset X$ such that for any $x \in X$ there is $\gamma \in \Gamma$ such that $ G_\gamma(x) \in S$.

The standard example of a cobounded quasi-action arises when $\Gamma$ is a  group with   a left invariant  metric (for example, a  finitely generated group  with a word metric or a Lie group with a left invariant Riemannian metric) and $\phi: \Gamma \to X$ is a quasi-isometry with coarse inverse $\bar{\phi}$. Then $\gamma \mapsto
\phi\circ L_\gamma \circ \bar{\phi}$   defines a cobounded quasi-action   of $\Gamma$ on $X$, where $L_\gamma$ is the left translation of $\Gamma$ by $\gamma$.
 We note, however, that there  exist ``non-proper'' cobounded  quasi-actions and so these do not come from a quasi-isometry between a group and a metric space.

      A quasi-action  $\{G_\gamma|\gamma\in \G\}$  of $\G$ on a metric space $X$ is quasi-conjugate to a quasi-action   $\{G'_\gamma|\gamma\in \G\}$
        of $\G$ on another metric space $X'$ if there is a quasi-isometry
      $f: X\ra X'$  and a constant $C>0$   such that
       $d(G'_\gamma(f(x)), f(G_\gamma(x)))\le C$   for all $x\in X$  and all $\gamma\in \G$.

    The relation between   quasi-actions and uniform quasiconformal groups is through negative curvature.   A  self quasi-isometry of a  Gromov hyperbolic space $X$   induces a
      self quasiconformal map of the Gromov boundary  $\partial X$ of $X$ (equipped with   a  visual metric),
        and a  quasi-action of a group $\G$ on a Gromov hyperbolic space $X$
         induces a uniform quasiconformal group action of $\G$  on $\partial X$.
           A quasi-conjugate  between quasi-actions of $\G$ on two Gromov hyperbolic spaces $X$, $X'$     corresponds to  a quasiconformal  conjugation between  the   uniform quasiconformal actions of  $\G$ on $\partial X$   and   $\partial X'$.

  The standard Carnot group dilations of $N$ define an action of $\mathbb R$ on $N$.  Let
   $S=N\rtimes \mathbb R$ be the associated solvable Lie group.  Then $S$
   with  any left invariant Riemannian
      metric   is a Gromov hyperbolic space
      \cite{H74}.

    \begin{corollary}\label{quasiaction}
    Let $N$ be a Carnot group and $S=N\rtimes \mathbb R$ be the associated
    solvable Lie group.  There is a left invariant Riemannian metric $g_0$ on $S$ with the following property.
   Let $G$ be a  
      group that quasi-acts  on  $S$. If  the quasi-action is co-bounded, then
    the quasi-action    is quasi-conjugate to an isometric action of $G$ on $(S, g_0)$.
    \end{corollary}

    We next turn  to  uniform quasiconformal groups on more general nilpotent Lie groups.
     Let $N$ be  a  simply connected nilpotent Lie group with Lie algebra $\mathfrak n$
       and $D$  a derivation of $\mathfrak n$.  We say $(N, D)$ is a diagonal Heintze pair if
        $D$ has positive eigenvalues and is diagonalizable over $\mathbb R$.
          Let $(N, D)$ be  a diagonal Heintze pair.
  A distance   $d$  on $N$ is called $D$-homogeneous    if it is left invariant,   induces the manifold topology on $N$ and such that
 $d(e^{tD}x, e^{tD}y)=e^t d(x, y)$ for all $x, y\in N$ and $t\in \mathbb R$, where $\{e^{tD}|t\in \mathbb R\}$  denotes the automorphisms  of $N$ generated by the derivation $D$.   
   By  
    Theorem 2 of \cite{HSi90}, $D$-homogeneous    distances exist on $N$.
     It is easy to see that any two $D$-homogeneous distances on $N$ are biLipschitz equivalent.   We will always equip $N$ with a $D$-homogeneous distance. Hence it makes sense to
   speak of  a biLipschitz map of $N$ without specifying the  $D$-homogeneous distance.

  Let $(N, D) $ be a  diagonal Heintze pair. 
      Then there is a sequence of $D$-invariant Lie sub-algebras
     $\{0\}=\mathfrak n_0\subset \mathfrak n_1\subset \cdots\subset \mathfrak n_s= \mathfrak n$  with the following properties:   each $\mathfrak n_{i-1}$ is an ideal of $\mathfrak n_{i}$  with the quotient
     $\mathfrak n_{i}/{\mathfrak n_{i-1}}$ a Carnot Lie algebra;  $D$ induces a derivation $\bar D:  \mathfrak n_{i}/{\mathfrak n_{i-1}}  \ra
      \mathfrak n_{i}/{\mathfrak n_{i-1}} $  which is a multiple of the Carnot derivation of
       $\mathfrak n_{i}/{\mathfrak n_{i-1}}$. See Section \ref{bilip-diagonal} for more details.
      Let $N_i$ be the connected Lie subgroup of $N$ with Lie algebra $\mathfrak n_i$.  Then $N/{N_i}$  is a homogeneous manifold and the natural map
       $\pi_i:  N/{N_{i-1}}   \ra N/{N_i}$  is a fiber bundle with fiber the Carnot group $N_i/{N_{i-1}}$.
        We call the sequence of subgroups $0=N_0<N_1<\cdots<N_s=N$  the \emph{preserved subgroup sequence}.

         Let $d$ be  a $D$-homogeneous distance  on $N$.
       In general  $d$  does not induce any metric on the homogeneous   space $N/{N_i}$
         when $N_i$ is not normal in $N$.  Nonetheless,     it induces a metric on the fibers
        $ N_i/{N_{i-1}}$  of  $\pi_i:  N/{N_{i-1}}   \ra N/{N_i}$.
    Furthermore, every biLipschitz map $F$ of $N$ permutes the cosets of $N_i$ for each $i$.   Hence $F$ induces a   map $F_i: N/{N_i}  \ra N/{N_i}$  and a bundle map of  $\pi_i:  N/{N_{i-1}}   \ra N/{N_i}$.  The restriction of $F_{i-1}$ to  the  fibers of $\pi_i$ are biLipschitz maps of the Carnot group   $ N_i/{N_{i-1}}$ in the following sense.      For each $p\in N$, let $F_p=L_{F(p)^{-1}}\circ F\circ L_p$, where  $L_x$ denotes the left translation of $N$ by $x$.   Notice that the map
$(F_p)_{i-1}: N/{N_{i-1}}\ra N/{N_{i-1}}$ satisfies
$(F_p)_{i-1}(N_i/{N_{i-1}})=N_i/{N_{i-1}}$.  The statement above  simply means
  $(F_p)_{i-1}|_{N_i/{N_{i-1}}}:  N_i/{N_{i-1}}\ra  N_i/{N_{i-1}}$  is biLipschitz with respect to any  left invariant Carnot-Caratheodory metric on $N_i/{N_{i-1}}$.
See Section \ref{bilip-diagonal} for more details.

A group  $\G$ of homeomorphisms of a metric space $X$ is a 
 \emph{uniform quasisimilarity group}  if there is a constant $K\ge 1$ such that each $\gamma\in \G$ satisfies
  $(C_\gamma/K) d(x,y)\le d(\gamma(x), \gamma(y))\le C_\gamma K d(x,y)$ for some $C_\gamma>0$ and all $x, y\in X$.
     A bijection $f: X\ra X$ is a \emph{similarity} if there is some $L>0$ such that
      $d(f(x), f(y))=L d(x,y)$ for all $x,y\in X$.   

\begin{theorem}\label{foliatedtheorem}
Let  $(N,D)$ be  a diagonal Heintze pair  and
 $\G$ be a    
 uniform quasisimilarity group of $(N,D)$ that acts cocompactly on the space of distinct pairs of $N$ (or equivalently $\G$ a group that quasi-acts coboundedly on $S=N \rtimes_D \R$).
      Let $I=\{i| 1\le i\le s, \; \dim(N_i/{N_{i-1}})\ge 2\}$.
    Then there exists a biLipschitz map $F_0: N \to N$    and   a left invariant
     Carnot-Caratheodory metric  $d_i$ on  $N_i/{N_{i-1}}$ for each $i\in I$
    such  that     for each $p\in N$ and   each  $g\in F_0\G F^{-1}_0$, the map
      $(g_p)_{i-1}|_{N_i/{N_{i-1}}}:  (N_i/{N_{i-1}}, d_i)\ra  (N_i/{N_{i-1}}, d_i) $
          is a similarity.

\end{theorem}

Ideally one would like to conjugate  the group $\Gamma$  in   Theorem \ref{foliatedtheorem}   into a group of similarities of $N$ with respect to some $D$-homogeneous distance.   But this question is still open in general.
A positive answer was given in \cite{DFX}  in the case when the preserved subgroup sequence has only two terms $0<N_1<N$.   Its proof is  much more involved algebraically   and uses Theorem
\ref{foliatedtheorem}  as a crucial ingredient.

When $s\ge 2$,  \cite{CP17}   implies that every quasiconformal map of $\hat N=N\cup \{\infty\}$ fixes $\infty$ and  restricts to a biLipschitz map of $N$.
  From this it is easy to see that a uniform quasiconformal group of  $\hat N$ restricts to a
uniform quasisimilarity group of $N$.   Therefore  there is no loss of generality  in Theorem \ref{foliatedtheorem}
 in considering a uniform quasisimilarity group of $N$ instead of a uniform quasiconformal group of   $\hat N$.

 When $s=1$, then $N$ is Carnot and depending on $N$ not all quasiconformal maps of $\hat N$ are necessarily biLipschitz.
In this case,  Theorem \ref{foliatedtheorem} simply asserts that if the  quasiconformal group from Theorem \ref{tukia} happens to consist of biLipschitz maps then the conjugating map can be chosen to be biLipschitz.

 Theorem \ref{foliatedtheorem}   was proved in \cite{D10}  in the case when $N$ is a Euclidean group. In the case $N = \R$, this result can be found in the appendix of \cite{FM99} and no additional assumptions other than uniformity are needed on the group.

The case $N=\R$ is used for the last step
  in the  proof of quasi-isometric rigidity of SOL by   Eskin-Fisher-Whyte \cite{EFW12}, \cite{EFW13}, while the cases covered in \cite{D10} are used to prove quasi-isometric rigidity of higher rank generalizations of SOL by Peng in  \cite{P11a}, \cite{P11b}.
 Similarly Theorem \ref{foliatedtheorem}   played a crucial role in the proof of quasi-isometric rigidity of
     a class of  solvable groups  \cite{DFX}.

  The group $SOL= \R^2 \rtimes \R$ where the action of $\R$ scales by $e^t$ on the first coordinate and by $e^{-t}$ along the second. This action gives rise to two foliations by hyperbolic planes (which we view as $\R\rtimes \R$).
     More generally  a \emph{SOL-like group} is a semi-direct product $(N_1\times N_2)\rtimes \mathbb R$, where  $N_i$ is a simply connected nilpotent Lie group with Lie algebra $\mathfrak n_i$,  and  the action of  $\mathbb R$ on $N_1\times N_2$ is generated by a derivation $D=(-D_1, D_2)$ of $\mathfrak n_1\times \mathfrak n_2$  and $D_i$ is a derivation of $\mathfrak n_i$ whose eigenvalues have positive real part.   A SOL-like group is foliated by two families of negatively curved solvable groups $N_i\rtimes_{D_i}\mathbb R$ (these are called Heintze groups  and they are exactly those solvable Lie groups admitting left invariant Riemannian metrics with negative sectional curvature).
  The quasi-isometric rigidity result proved in \cite{DFX} are for SOL-like groups where $(N_i, D_i)$ is either of Carnot type or  has  a preserved subgroup sequence of length two.
We remark that the quasi-isometric rigidity for non-unimodular  SOL-like groups   where $(N_i, D_i)$ is  of  Carnot Type is included in Theorem C,  \cite{Fe22}  and that proof does not use  Tukia type theorems.

 {\bf{Acknowledgments.}}
T. Dymarz was supported by NSF career  grant 1552234.  D. Fisher was supported  by  NSF grants DMS-1906107 and DMS-2246556. 
X. Xie   was partially  supported by
    Simons Foundation grant \#315130.   
 X. Xie   would  like to thank  the department of mathematics,   University of Wisconsin at Madison   for financial support  during his visit there   in February 2020.
We would also like to thank  Tom Ferragut   for   useful comments on an earlier version.

 \section{Preliminaries}\label{preli}

\subsection{Carnot groups}

Let $N$ be a Carnot group with Lie algebra $\mathfrak n=V_1\oplus\cdots \oplus V_k$.
The exponential map ${\text{exp}}:   \mathfrak n\ra N$ is a diffeomorphism.  We will identify $N$ and $\mathfrak n$ via the exponential map.
  For any $t>0$,   the Carnot group dilation  $\delta_t:   \mathfrak n\ra \mathfrak n$  is  given by
  $\delta_t(\sum_{j=1}^k x_j)=\sum_{j=1}^k t^j x_j,$ with $x_j\in V_j$.    They are similarities w.r.t any left invariant
    Carnot-Caratheodory metric $d$:   $d(\delta_t(x), \delta_t(y))=t\, d(x,y)$  for any $x, y\in \mathfrak n$.
    The determinant of $\delta_t: \mathfrak n\ra \mathfrak n$ is $t^Q$, where $Q=\sum_{j=1}^k jm_j$ is the homogeneous dimension of $N$.

Let $\{e_{jl}:   1\le l\le m_j\}$ be a basis for $V_j$, $1\le j\le k$.
 Let $n=\dim \mathfrak n$.
  Then the  map $\mathbb R^n\ra N$ given by $(x_{jl})\to \text{exp}(\sum_{j,l}x_{jl}e_{jl})$  is a diffeomorphism and the push-forward of the Lebesgue measure under this  map is a Haar measure on $N$.    We shall use the notation $|A|$
     for the Haar measure of a subset $A\subset N$.
  Define  a function  $\rho: \mathfrak n\ra [0, \infty)$  by
 $$\rho\left(\sum_{j,l}x_{jl}e_{jl}\right)=\sum_{j=1}^k \sum_{l=1}^{m_j}|x_{jl}|^{\frac{1}{j}}.$$
   Then $\rho(\delta_t(x))=t\, \rho(x)$  for all $x\in \mathfrak n$.
         There is an associated  ``distance" $d_\rho$  given by
  $d_\rho(x,y)=\rho((-x)*y)$.
     It is easy to see that   $d_\rho$ is left invariant  and satisfies
       $d_\rho(\delta_t(x), \delta_t(y))=t\, d_\rho(x,y)$.    Hence  $d_\rho$ is biLipschitz equivalent   with any left invariant Carnot-Caratheodory metric  $d$:
 there exists some constants  $L\ge 1$ such that
  $d(x,y)/L\le d_\rho(x,y)\le L d(x,y)$ for all $x, y\in N$.

       Let $V_1$ be equipped with an inner product and  we may assume 
         $\{e_{1l}: 1\le l\le m_1\}$   is  an orthonormal basis for $V_1$.   Let  $X_{1l}$ be the left invariant vector field on $N$
          determined by $e_{1l}$.             For any smooth function $u: U\ra \R$ defined on an open subset of $N$,
                   define the horizontal gradient $\triangledown u$   of $u$  by:
                    $$\triangledown u=\sum_{l=1}^{m_1} (X_{1l}u)X_{1l}.$$
                      The length of $\triangledown u$  is:
                      $$|\triangledown u|=\sqrt{\sum_{l=1}^{m_1} (X_{1l}u)^2}.$$

Let   $d$ be the  left invariant  Carnot-Caratheodory metric  on $N$ determined by the inner product on $V_1$,
   let    $U, V\subset N$    be open subsets, and $f: U\ra V$ a   homeomorphism.
For $x\in U$ and $r>0$ with $B(x,r)\subset U$, let
 $$L_f(x,r)=\sup_{d(y, x)\le r} d(f(x), f(y)),   \;\;\; l_f(x,r)=\inf_{d(y, x)\ge r} d(f(x), f(y)).$$
     Define
 $$K(f,x)=\limsup_{r\ra 0}\frac{L_f(x,r)}{l_f(x,r)}.$$
 We call $f$ $K$-quasiconformal  for some $K\ge 1$ if $K(f,x)\le K$ for a.e. $x\in U$.

Let $U, V\subset N$ be open subsets, and $f: U\ra V$ a quasiconformal map. By   \cite{P89}, $f$ is Pansu differentiable a.e. and the Pansu differential   $Df(x)$ is a graded automorphism for a.e.    $x\in U$.
 Recall that an automorphism $A: \mathfrak n\ra \mathfrak n$ is graded if it commutes with $\delta_t$ for all $t>0$;
  equivalently $A$ is graded if  $A(V_j)=V_j$ for each $1\le j\le k$.    Furthermore,   $f$ is absolutely continuous w.r.t.
       the  Haar measure, see \cite{P89}.
     For any $x\in U$, define
     $$L_f(x)=\limsup_{y\ra x}\frac{d(f(x), f(y))}{d(x,y)},\;\;\;  l_f(x)=\liminf_{y\ra x}\frac{d(f(x), f(y))}{d(x,y)}.$$
            By  Lemma 3.3 in \cite{CC06},   for a.e. $x\in U$, we have
     $$L_f(x)=\max\{Df(x)X:  X\in V_1,  |X|=1\},$$
      $$l_f(x)=\min\{Df(x)(X):  X\in V_1,  |X|=1\},$$
       and $K(f,x)=\frac{L_f(x)}{l_f(x)}$.    For any $x\in U$, the volume derivative of $f$ at $x$ is:
       $$f'(x)=\lim_{r\ra 0} \frac{|f(B(x,r))|}{|B(x,r)|},$$
    which exists a.e. and  is a.e. finite.          By (4.1) in \cite{CC06},   $l_f(x)^Q\le f'(x)\le L_f(x)^Q$ for a.e. $x\in U$,
      where $Q$ is the homogeneous dimension of $N$.

\subsection{Homogeneous distances on nilpotent Lie groups}\label{homo distance}

Let $(N,D)$ be a   diagonal Heintze pair.
     Let   $0<\lambda_1<\cdots <\lambda_r$ be the distinct eigenvalues of $D$ and 
     $\mathfrak n=\oplus_j V_{\lambda_j}$ be the decomposition of $\mathfrak n$ into the direct sum of eigenspaces of $D$.
  An inner product $\left<\cdot , \cdot \right>$ on $\mathfrak n$ is called a $D$-inner product if
  the eigenspaces corresponding to distinct eigenvalues are perpendicular with respect to
  $\left<\cdot , \cdot \right>$.  By the construction in Theorem 2 of \cite{HSi90}, given any $D$-inner product $\left<\cdot , \cdot \right>$ on $\mathfrak n$, there is a $D$-homogeneous distance $d$ on $N$ such that
  $d(0,x)=\langle x,x\rangle^{\frac{1}{2\lambda_j}}$ for $x\in V_{\lambda_j}$.

        For computational purposes, we also define a function $\rho$ that is biLipschitz equivalent to a $D$-homogeneous
        distance
         $d$.  For any  $D$-inner product $\left<\cdot , \cdot \right>$ on  $\mathfrak n$  define a ``norm'' on $\mathfrak n$ by
                 $$||v||=\sum_i |v_i|^{\frac{1}{\lambda_i}} ,$$
                  where $v=\sum_i v_i$ with $v_i\in V_{\lambda_i}$.
                  Then   define $\rho$ by
                  $\rho(x,y)=||x^{-1}*y||$.  We identify $\mathfrak n$ and $N$.
                   Clearly $\rho$ is left invariant, induces the manifold topology and satisfies
                   $\rho(e^{tD}x, e^{tD}y)=e^t \rho(x, y)$ for all $x, y\in \mathfrak n$ and $t\in \mathbb R$.   It follows that for any $D$-homogeneous  distance $d$ on $\mathfrak n$,  there is a constant $L\ge 1$ such that   $d(x,y)/L\le \rho(x,y)\le L\cdot d(x,y)$ for all $x, y\in \mathfrak n$.
                       The explicit formula for $\rho$ will make the calculations much easier.

 \begin{lemma}\label{bilip auto}
 Let $\phi$ be an  automorphism of $N$.
 Then $\phi$ is  biLipschitz  if and only if  $d\phi$ is
  ``layer preserving''; that is, $d\phi(V_{\lambda_j})=V_{\lambda_j}$ for each $j$.
 \end{lemma}

        \begin{proof}
        First suppose $\phi$ is biLipschitz. Let $0\not=v\in V_{\lambda_j}$ and write
        $d\phi(v)=\sum_i x_i$ with $x_i\in V_{\lambda_i}$.
        Then $d\phi(tv)=\sum_i{tx_i}$.
           We have $$\rho(0, tv)= |v|^{\frac{1}{\lambda_j}}|t|^{\frac{1}{\lambda_j}}$$
             and
             $$\rho(0, d\phi(tv))=\sum_i |x_i|^{\frac{1}{\lambda_i}}|t|^{\frac{1}{\lambda_i}}.$$
             The biLipschitz condition implies $x_i=0$ when $i\not=j$ by letting $t\ra \infty$ or $t\ra 0$.

        Conversely assume $d\phi$ is layer preserving.
          Then there is some constant $C\ge 1$ such that
          \begin{equation}           \label{phi(v)}
          |v|/C\le |d\phi(v)|\le C |v|
          \end{equation}
            for all $v\in V_{\lambda_j}$, $\forall j$.  Now let $v\in \mathfrak n$.  Write
          $v=\sum_j v_j$ with $v_j\in V_{\lambda_j}$.  Then
          $d\phi(v)=\sum_j d\phi(v_j)$.
           We have $\rho(0, v)=\sum_j|v_j|^{\frac{1}{\lambda_j}}$ and
            $\rho(0,d \phi(v))=\sum_j|d\phi(v_j)|^{\frac{1}{\lambda_j}}$.
        Now the claim follows from (\ref{phi(v)}).

                 \end{proof}

An automorphism $\phi$  of $N$ is called graded if it satisfies the condition in Lemma
\ref{bilip auto}.   We denote by $\text{Aut}_g(N)$  the group of graded automorphisms of $N$.

For any $D$-homogeneous distance $d$ on $N$, let $\text{Sim}(N,d)$ be the group of similarities  of $(N,d)$.    By       Theorem 1.2 in \cite{KLD17}
  and  Lemma \ref{bilip auto},    $\text{Sim}(N,d)$  has the structure  $\text{Sim}(N,d)=N\rtimes (\mathbb R\times K)$, where $\mathbb R$ acts on $N$ by the automorphisms $\{e^{tD}|t\in \mathbb R\}$  and $K\subset \text{Aut}_g(N)$    is a compact  Lie  subgroup.   Given two $D$-homogeneous distances  $d_1$, $d_2$   on $N$, although $(N,d_1)$ and $(N,d_2)$ are biLipschitz,  a similarity  of $ (N,d_1)$ in general is not a similarity of   $(N,d_2)$ and so
their associated similarity groups can be very different;  see    Section \ref{example}  for an example in the  case of  Carnot groups.  The following is  a notion of $D$-homogeneous distance with the largest similarity group.  It is similar to  Definition  0.2 in
\cite{GJ19}.

\begin{definition}\label{max-def}
Let $(N,D)$ be a diagonal  Heintze pair. A $D$-homogeneous distance  $d_0$ on $N$ is maximally symmetric (with respect to similarities) if for any
 $D$-homogeneous distance  $d$ on $N$, there is a biLipschitz automorphism $\phi$ of $N$
  such that $\phi\text{Sim}(N,d)\phi^{-1} \subset \text{Sim}(N,d_0)$.
\end{definition}

\begin{lemma}\label{maximal}
Let $(N,D)$ be a diagonal   Heintze pair.  Then $N$ admits
   a  maximally symmetric $D$-homogeneous distance.

\end{lemma}

\begin{proof}
Since $N$ is simply connected, $\text{Aut}_g(N)$  can be  identified with the group of graded automorphisms     $\text{Aut}_g(\mathfrak n)$  of $\mathfrak n$.
 It is easy to see that $\text{Aut}_g(\mathfrak n)$  is a real algebraic variety and so has    only a finite number of connected components by Whitney's theorem \cite{Wh57}.   Therefore $\text{Aut}_g(N)$   is  a Lie group with finitely many components.

   Let $K_0$ be a maximal compact subgroup of $\text{Aut}_g(N)$.
      Recall $\mathfrak n=\oplus_j V_{\lambda_j}$.
     For each $j$ let
    $\left<\cdot ,\cdot \right>_j$   be a $K_0$-invariant inner product  on $V_{\lambda_j}$.
        Let $\left<\cdot , \cdot \right>$ be the inner product on $\mathfrak n$ that agrees with
         $\left<\cdot ,\cdot \right>_j$  on $V_{\lambda_j}$  such that $V_{\lambda_i}$ and $V_{\lambda_j}$ are perpendicular to each other for $i\not=j$.  Let $d$ be a   $D$-homogeneous distance on $N$ associated to this inner product.
           Although  $d\phi$ is a linear isometry of  $(\mathfrak n, \langle,\rangle)$  for any $\phi\in K_0$, it is not clear that
            $\phi$
            is an isometry of $(N,d)$.
             Let $m$ be the normalized Haar measure on $K_0$.  Define a new distance $d_0$ on $N$ by  $d_0(x,y)=\int_{K_0} d(k(x), k(y))dm(k)$.   Now it is easy to check that $d_0$ is a
      $K_0$-invariant         $D$-homogeneous distance on $N$ associated to   $\langle,\rangle$
        and  $\text{Sim}(N,d_0)=N\rtimes (\mathbb R\times K_0)$.

        Now let $d$ be an arbitrary $D$-homogeneous distance on $N$. As observed above,
           $\text{Sim}(N,d)=N\rtimes (\mathbb R\times K)$, where $K$ is a  compact subgroup of $\text{Aut}_g(N)$.   Since $\text{Aut}_g(N)$  has only a finite number of components, there is some $\phi\in \text{Aut}_g(N)$  such that $\phi K\phi^{-1}\subset K_0$.  Since $N$ is normal in $N\rtimes \text{Aut}(N)$ and $\phi$ is graded we have
           $\phi\text{Sim}(N,d)\phi^{-1} \subset \text{Sim}(N,d_0)$.

\end{proof}

Let $G$ be a connected Lie group with a left invariant distance $d$ that induces the manifold topology, and $H$ a  closed normal  subgroup of $G$.
 We define a distance on $G/H$ by $\bar d(xH, yH)=\inf\{d(xh_1, yh_2)| h_1, h_2\in H\}$.     Then $\bar d$ is a left invariant distance on $G/H$ that induces the manifold topology and the quotient map $(G, d) \to (G/H, \bar d)$ is $1$-Lipschitz.
Since $H$ is normal,   we have $\bar d(xH, yH)=d(xh_1, yH)=d(yh_2, xH)=d_H(xH, yH)$ for any $h_1, h_2\in H$, where
 $d_H$ denotes the Hausdorff distance.
If $F$ is a biLipschitz map of $(G, d)$ that permutes the cosets of $H$, then $F$ induces a   biLipschitz map $\bar F: (G/H, \bar d) \to (G/H, \bar d)$   with the same biLipschitz constant as $F$.

Let $(N, D)$ be a diagonal
  Heintze pair and $d$ a $D$-homogeneous distance on $N$. Assume $\mathfrak w$ is an ideal of $\mathfrak n$ such that
 $D(\mathfrak w)\subset
\mathfrak w$.   Then $D$ induces a derivation $\bar D$ of $\mathfrak n/\mathfrak w$   and  $(N/W, \bar D)$   is  also  a  diagonal   
 Heintze pair, where $W$ is the Lie subgroup of $N$ with Lie algebra $\mathfrak w$.
     In this case, the distance $\bar d$ on $N/W$ induced by $d$ is a $\bar D$-homogeneous distance.

\subsection{Homogeneous manifolds with negative curvature}

Let $N$ be a Carnot group with Lie algebra $\mathfrak n=V_1\oplus \cdots \oplus V_k$.
  The standard dilations $\delta_t$ of $N$ define an action of $\mathbb R$    on $N=\mathfrak n$:
  $$t\cdot x=\delta_{e^t}(x)\;\;  \text{for}  \;\; t\in \mathbb R\;\;  \text{and} \;\; x\in \mathfrak n.$$
  Let
$S= N \rtimes \R $   be     the associated semi-direct product.
  Then  $S $ is a   solvable Lie group.
 By \cite{H74} $S$   admits a left invariant Riemannian metric with negative sectional curvature.
  For $x_0\in N$, the path $c_{x_0}: \R\ra S$, $c_{x_0}(t)=(x_0, t)$, is a geodesic in $S$.   We call   $c_{x_0}$
  a    vertical geodesic.
 All  vertical geodesics are asymptotic as $t \to \infty$, and so they determine a point $\infty$  in the ideal boundary. If $t \to -\infty$ all vertical geodesics diverge from one another. We call such geodesics downward oriented.
  Every geodesic ray in $S$ is asymptotic to either an upward oriented vertical geodesic or a downward oriented vertical geodesic.   It follows that
  the ideal boundary  $\partial  S$  of $S$   can be  identified with
 $\hat N=N\cup \{\infty\}$, where points   in $N$ correspond to downward oriented vertical geodesics.

\subsection{Sphericalization of metrics and measures}

Let $(X, d)$ be an unbounded metric space and $p\in X$ a base point.
    Let $\infty$ be a  point not in $X$ and set $\hat X=X\cup \{\infty\}$.
      The  sphericalized metric  $\hat d_p$   of $d$ relative to the base point $p$   is   a  metric
         on $\hat X$ satisfying:
          $$\frac{d(x,y)}{4(1+d(p,x))(1+d(p,y))}\le \hat d_p(x,y)\le  \frac{d(x,y)}{(1+d(p,x))(1+d(p,y))}\;\;\text{for}\;\; x,y\in X,$$
and  $\hat d_p(x, \infty)=\frac{1}{(1+d(p,x))}$.    Furthermore,   a     proof  similar to that of
  Proposition 4.1 in   \cite{BHX}
       shows that  the  identity map
   $\text{id}:   (X,d)\ra (X, \hat d_p)$ is $1$-quasiconformal.  It follows that a homeomorphism
    $f:  ( \hat X, \hat d_p)\ra ( \hat X, \hat d_p)$ is quasiconformal   ($1$-quasiconformal) iff the restriction
    $f:  (X\backslash\{f^{-1}(\infty)\},    d)  \ra (X\backslash\{f(\infty)\}, d)$  is quasiconformal ($1$-quasiconformal).
            We shall use this observation when we study  quasiconformal maps of $\hat N$.

              Let $(X, d, \nu) $ be a metric measure space, with $\nu$ a Borel regular measure.
                 The metric measure space   $(X, d, \nu)$ is  $\alpha$-Alhfors regular  for some $\alpha>0$   if   there is some constant $C>0$
                  such that $r^\alpha/C\le \nu(B(x, r)) \le C r^\alpha$ for all balls $B(x, r)$ with radius $0<r\le \text{diam}(X,d)$.
  Now let $(X, d, \nu)$ be a  $\alpha$-Alhfors regular metric measure space with $(X, d)$ unbounded and $p\in X$ a base point.
   Li-Shanmugalingam   \cite{LS}  defined   a measure $\hat \nu_p$ on $(\hat X, \hat d_p) $ which is also
    $\alpha$-regular  and showed that  $(X, d, \nu)$ supports a $p$-Poincare inequality if and only if
     $(\hat X, \hat d_p,   \hat \nu_p)$   supports a $p$-Poincare inequality (see Theorem 1.0.1. in \cite{LS}) . Furthermore,  the two measures $\nu$ and $\hat \nu_p$ are comparable on any bounded subset of $(X, d)$; that is, for any bounded subset $A\subset (X, d)$, there is a constant $C\ge 1$ such that
     $\nu(E)/C\le \hat \nu_p(E)\le C\nu(E)$ for any $E\subset A$.
We shall use this in the case of a Carnot group $N$ equipped with a left invariant  Carnot-Caratheodory metric and  the   Lebesgue measure.

 \section{Lemmas on quasiconformal maps of Carnot groups}\label{lemmaonQC}

In this section we collect some  results on  quasiconformal maps of Carnot groups. These will be used in
  Section \ref{conjugate} for the proof of Theorem \ref{tukia}.

Let $N$ be a Carnot group equipped with a left invariant Carnot-Caratheodory metric and $\hat N=N\cup \{\infty\}$ its one-point compactification.
A ring  in $\hat N$    is a connected open subset $R\subset \hat N$ whose complement
 has two connected components.
 We always work with rings $R$  satisfying $\infty\notin \bar R$.
   Let $C_0$ and $C_1$ be the two components of $\partial R$.   An admissible function for $R$ is a  $C^\infty$ function $u: N\ra \mathbb R$
    with $u|_{C_0}=0$ and $u|_{C_1}=1$.     The conformal capacity of  $R$ is:
         $$C(R)=\inf_{u}\int_N |\triangledown u|^Q dx,$$
          where $Q$ is the homogeneous dimension of $N$,     $u$ ranges over all admissible functions for $R$,
            and the integral is with respect to the Lebesgue measure.
             The above infimum remains the same  if   we  enlarge the class of admissible functions to include all
              Sobolev functions in $C(\bar R)\cap W^{1, Q}(R)$,   see  Proposition 11 in \cite{KR}, where the  proof is valid for all Carnot groups.  In particular,   if $f:\hat N\ra \hat N$ is  quasiconformal,
                $R$
               is a ring    with $\infty, f^{-1}(\infty)\notin \bar R$    and $u: N\ra \mathbb R$ is an admissible function for $f(R)$  as defined  above, then
                $u\circ f$ is admissible for $R$ in the generalized sense and so can be used in the estimate for $C(R)$.

          Let $\Gamma$ be a curve family in $N$.
           A non-negative Borel function $\rho: N\ra [0, +\infty]$ is an  admissible  function of $\Gamma$
                if
            $\int_\gamma \rho\, ds\ge 1$ for every locally rectifiable curve $\gamma$ in $\Gamma$.
              The conformal modulus  of $\Gamma$ is:
               $$M(\Gamma)=\inf_\rho \int_N \rho^Q dx,$$
                where $Q$ is the homogeneous dimension of $N$ and   $\rho$ varies over all admissible functions of
                 $\Gamma$.

          It is a classical result that   for any ring $R$ in  $\mathbb R^n$,  the capacity agrees with the modulus,
           $C(R)= M(\Gamma(R, C_0, C_1))$, where
            $\Gamma(R, C_0, C_1)$ is the  collection of  curves in $R$  joining $C_0$ and $C_1$.
             This result has been generalized to the case of Carnot groups by Markina  \cite{M}.

             The following result says that  a quasiconformal map preserving conformal capacity  of rings must be a conformal map.
              The  actual  assumption is slightly weaker.

             \begin{lemma}\label{capacity1qc}
             Let $f: \hat N\ra \hat N$ be a  quasiconformal  map.  Suppose
               $C(R)=C(f(R))$ for any ring $R$ in $\hat N$  satisfying $\infty,  f^{-1}(\infty)\notin \bar R$.
                    Then
                 $f$ is   conformal.

             \end{lemma}

             \begin{proof}
                 Let $p\in N\backslash\{ f^{-1}(\infty)\} $ be a point   such that $f$ is Pansu differentiable at $p$ and the Pansu differential
                   $Df(p)$ is a
                  graded   automorphism of $N$.
                    Then the maps $   \delta_{\frac{1}{t}}\circ    L_{f(p)^{-1}}\circ f\circ L_p\circ \delta_t$   converges to
                     $Df(p)$ uniformly on compact subsets  as $t\ra +0$.
                   Since left translations and  the standard Carnot dilations are conformal and so preserve
                        the capacities,
                     all the maps $   \delta_{\frac{1}{t}}\circ    L_{f(p)^{-1}}\circ f\circ L_p\circ \delta_t$  preserve capacities.  Now the continuity of capacities implies the limiting map $Df(p)$
                     also preserves the capacity.   Hence we may assume $f$ is a graded automorphism of $N$
                     that preserves capacity of rings.   By the result of Markina \cite{M}  cited above,
                      $$M(\Gamma(R, C_0, C_1))=M(\Gamma(f(R), f(C_0), f(C_1)))$$
                         for any ring $R$.
                 We need to show that $f$ is a similarity.

               The proof below is a modification of the proof of Theorem 36.1 in \cite{V}.

               We first  construct the  rings that we will use.
               For this we identify $N$ with its Lie algebra $\mathfrak n=V_1\oplus \cdots \oplus V_k$.
                We fix an inner product on $\mathfrak n$ so that   the $V_j$'s are perpendicular to each other and  that  on $V_1$ it
                  agrees with the inner product on $V_1$ that defines the left invariant Carnot-Caratheodory metric on $N$.
                Denote $m_j=\dim(V_j)$.    By the singular value decomposition, there exist  orthonormal bases
                 $\{e_{j1}, \cdots, e_{j m_j}\}$  and
                 $\{\tilde e_{j1}, \cdots, \tilde e_{j m_j}\}$                   of $V_j$ and positive numbers $\lambda_{j1}\ge \cdots \ge \lambda_{jm_j}$  such that
                  $f(e_{jl})=\lambda^j_{jl} \tilde e_{jl}$.

                 Let
                 $$C_0=\left\{\sum_{j,l}  x_{jl}e_{jl}:    x_{11}=0,   |x_{1l}|\le 1 \;\;\text{for}\;\;  l=2, \cdots, m_1,
                 \;\; |x_{jl}|\le 1\;\;\text{for}\;\; j\ge 2, 1\le l\le m_j\right\}.$$
                   We will construct a   rectangular box    $Y$ that contains $C_0$ in its interior  $\mathring Y$  and the ring will be $R=\mathring Y\backslash{C_0}$.

             \noindent
         {\bf{Claim}}: For each pair $(j, l)$, $2\le j\le k$, $1\le l\le m_j$,  there is a polynomial $P_{jl}(\delta)$ without constant term  satisfying the following property:
            for $0<\delta\ll 1$, let
            $$Y= \left\{\sum_{j,l}  x_{jl}e_{jl}:    |x_{11}|\le \frac{\delta}{\lambda_{11}},   |x_{1l}|\le 1+\frac{\delta }{\lambda_{1l} } \;\;\text{for}\;\;  2\le l\le  m_1,
                 \;\; |x_{jl}|\le 1+P_{jl}(\delta)\;\;\text{for}\;\; j\ge 2, 1\le l\le m_j\right\},$$
                  then   $d(C_0, \partial Y)=\frac{\delta}{\lambda_{11}}$  and
                   $d(f(C_0), \partial f(Y))=\delta$.

                          We will   first assume the claim and finish the proof of the lemma, and then prove the claim.
                          We  consider the  ring $R=\mathring Y \backslash{C_0}$.
           Note $\partial R$ has two components: $C_0$ and $C_1:=\partial Y$.

           Let $\Gamma_1$  and $\Gamma_2$  be the families  of curves  defined by:
            $$\Gamma_1=\left\{x*(0, \frac{\delta }{\lambda_{11}})e_{11}:   \;\; x\in C_0\right\},$$
              $$\Gamma_2=\left\{x*(-\frac{\delta }{\lambda_{11}}, 0)e_{11}: \;\;  x\in C_0\right\}.$$
                The claim   $d(C_0, \partial Y)=\frac{\delta}{\lambda_{11}}$  and the definition of $Y$ implies
                 $\Gamma_1, \Gamma_2\subset  \Gamma:=\Gamma(R, C_0, C_1)$.
               The curves in $\Gamma_1$  and $\Gamma_2$ are  respectively  contained in
                the disjoint Borel sets $C_0*(0, \frac{\delta }{\lambda_{11}})e_{11}$  and
                                   $C_0*(-\frac{\delta }{\lambda_{11}}, 0)e_{11}$.
                                     This implies $M(\Gamma)\ge M(\Gamma_1)+M(\Gamma_2)$, see  Theorem 6.7 in \cite{V}.
                                   The standard calculation in
                                    quasiconformal analysis (see 7.2 in \cite{V}) shows that
                                     $$M(\Gamma_1)=M(\Gamma_2)
                                     =\frac{|C_0*[0, \frac{\delta }{\lambda_{11}}]e_{11}|}{(\frac{\delta }{\lambda_{11}})^Q}=\frac{\mathfrak L^{n-1}(C_0)\cdot (\frac{\delta }{\lambda_{11}})}{(\frac{\delta }{\lambda_{11}})^Q},$$
                                       where $\mathfrak L^{n-1}(C_0)$  is the $(n-1)$-dimensional Lebesgue
                                          measure of $C_0$   and $n=\dim(\mathfrak n)$.
                So    we have
               \begin{align}\label{e200}
               M(\Gamma)\ge M(\Gamma_1)+M(\Gamma_2)=\frac{2\, \mathfrak L^{n-1}(C_0)}{(\frac{\delta }{\lambda_{11}})^{Q-1}}.
               \end{align}

             By the claim  the minimal distance between the two boundary components $f(C_0)$ and $f(C_1)$   of     $f(R)$ is  $\delta $.
                  Now by  Theorem 7.1 in \cite{V},
                with $\tilde \Gamma=\Gamma(f(R), f(C_0), f(C_1))$ and $J$ the absolute value of the determinant of $f$,
                                \begin{align}\label{e201}
                M(\tilde \Gamma)\le  \frac{|f(R)|}{\delta ^Q}=\frac{J\cdot  |R|}{\delta ^Q}
                =\frac{ J\cdot \mathfrak L^{n-1}(F)\cdot 2\frac{\delta }{\lambda_{11}} }{\delta ^{Q}},
                \end{align}
                 where $F=\{\sum_{j,l}  x_{jl}e_{jl}\in Y: x_{11}=0\}$.

               Since $M(\Gamma)=M(\tilde \Gamma)$,   (\ref{e200}) and (\ref{e201}) imply
                $$\lambda_{11}^Q\le J\cdot \frac{\mathfrak L^{n-1}(F)}{\mathfrak L^{n-1}(C_0)}.$$
                  Since $\mathfrak L^{n-1}(F)/\mathfrak L^{n-1}(C_0)  \ra 1$ as $\delta\ra 0$, we  get
                   $\lambda_{11}^Q\le J$.
                   On the other hand,  by  Lemma 3.3 in \cite{CC06},
                 $f(B(0, 1))\subset  B(0, \lambda_{11})$, which implies
                  $$J\cdot |B(0,1)|=|f(B(0,1))|\le |B(0, \lambda_{11})|=\lambda_{11}^Q |B(0,1)|$$
                    and so $J\le  \lambda_{11}^Q$.
                     Hence                   $J=  \lambda_{11}^Q$. This implies  $f$ maps the ball  $B(0,1)$  onto the ball
                $B(0, \lambda_{11})$ and so must be a similarity.

                \noindent
               {\bf{Proof of the claim:}}
                  We only  write down the proof for  $d(f(C_0),  \partial f(Y))=\delta$, as the proof for
                   $d(C_0, \partial Y)=\frac{\delta}{\lambda_{11}}$ is very similar.
                 Let $P_{jl}$ be a polynomial without constant term and $Y$ be as defined in the claim.
                By the formula for $f$ we have
                 $$f(C_0)=\left\{\sum_{j,l}  y_{jl}\tilde e_{jl}:    y_{11}=0,   |y_{1l}|\le \lambda_{1l} \;\;\text{for}\;\;  l=2, \cdots, m_1,
                 \;\; |y_{jl}|\le \lambda_{jl}\;\;\text{for}\;\; j\ge 2\right\}$$
                 and
                 $$f(Y)= \left\{\sum_{j,l}  y_{jl} \tilde e_{jl}:    |y_{11}|\le \delta,   |y_{1l}|\le    \lambda_{1l}+\delta  \;\;\text{for}\;\;  l=2, \cdots, m_1,
                 \;\; |y_{jl}|\le \lambda_{jl}+\lambda_{jl}P_{jl}(\delta)\;\;\text{for}\;\; j\ge 2\right\};$$
                  First note that $0\in f(C_0)$,    $\delta \tilde e_{11}\in \partial f(Y)$  and
                    $d(0, \delta \tilde e_{11})=\delta$. So $d(f(C_0), \partial f(Y))\le \delta$.
                It suffices to show for suitable choices of $P_{jl}$ we have
                 $N_\delta(f(C_0))\subset f(Y)$.

                Let $w\in N_\delta(f(C_0))$.  Then  we can write $w=y*z$ with
                $y\in f(C_0)$ and $d(0,z)\le \delta$.    By the BCH formula we have
                 $w=y+z+\frac{1}{2}[y,z]+\cdots$.  Write $w=\sum_{j,l}w_{jl}\tilde e_{jl}$,
                  $y=\sum_{j,l}y_{jl}\tilde e_{jl}$, $z=\sum_{j,l}z_{jl}\tilde e_{jl}$.
                   As $d(0,z)\le \delta$, we have $|z_{jl}|\le \delta^j$.
                   As $y_{11}=0$ and $|y_{1l}|\le \lambda_{1l}$ for $l\ge 2$ we have
                    $|w_{11}|\le \delta$ and $|w_{1l}|\le \lambda_{1l}+\delta$  for $l\ge 2$.
                      For $j\ge 2$ we have $w_{jl}=y_{jl}+Q_{jl}$, where $Q_{jl}$ is a polynomial of $z_{st}$ and $y_{st}$ with $s\le j$ such that each monomial in
                    $Q_{jl}$  has at least one $z_{st}$ as a factor.  As $|y_{st}|\le \lambda_{st}$ and
                     $|z_{s,t}|\le \delta^s$, we see that there is a polynomial $P_{j,l}$   with no constant term such that  $|Q_{jl}|\le \lambda_{jl}P_{jl}(\delta)$.   This finishes the proof of the claim.

             \end{proof}

             In Lemma \ref{uniformdensity}  and Lemma \ref{holder}:  we fix an inner product on $\mathfrak n$ so that the $V_j$'s are perpendicular
              to each other;    this
              determines a left invariant Carnot-Caratheodory metric  $d_{CC}$ on $N$ and            a  Haar measure on $N$;   the metric on $\hat N$ will be the sphericalized metric    $\hat d_{CC}$  of $d_{CC}$ with respect to the origin  and the measure on $\hat N$ will be the sphericalized measure  $m$ of the Haar measure with respect to the origin.   See the end of Section \ref{preli} for more details.

             Lemma \ref{uniformdensity}  and Lemma \ref{holder}
             hold  for quasisymmetric maps on
               Alhfors regular metric measure spaces that satisfies  a Poincare inequality:
                 their  proofs use    only the reverse Holder inequality which holds in such generality, see
                   Theorem   7.11 of \cite{HK}.

 \begin{lemma}\label{uniformdensity}
 Let  $f: (\hat N, \hat d_{CC}) \ra (\hat N, \hat d_{CC})$  be an $\eta$-quasisymmetric map.
    Then there exist constants    $a, b>0$  depending only on $N$ and $\eta$ that satisfy the following.
      For any ball  $B_1\subset \hat N$ with center $x$, let $B_2$ be the largest ball in $f(B_1)$  with center $f(x)$.   Then for any measurable $E\subset \hat N$,
      $$\frac{m(f(E)\cap B_2)}{m(B_2)}\le a \left(\frac{m(E\cap B_1)}{m(B_1)}\right)^b.$$

 \end{lemma}

 \begin{proof}  By \cite{J},   $N$ with the Carnot-Caratheodory metric and the   Lebesgue measure supports a $1$-Poincare inequality.
    Then  Theorem 1.0.1. of \cite{LS}  implies $\hat N$ with the sphericalized metric and sphericalized measure also supports a $1$-Poincare inequality.
 By  Theorem 7.11 in  \cite{HK},  there is some  $\epsilon>0$ and a constant $C$ depending only on
  $N$ and $\eta$    such that
  $$\left(\frac{1}{m(B)}\int_B \mu_f^{1+\epsilon} \;dm\right)^{\frac{1}{1+\epsilon}}\le C\left(\frac{1}{m(B)}\int_B \mu_f \;dm\right) $$
   for all balls $B$ in $\hat N$. Here    $\mu_f$ is the volume derivative  of $f$ defined  by:
    $$\mu_f(x)=\lim_{r\ra 0}\frac{m(f(B(x, r)))}{m(B(x,r))},$$
        which exists  and is finite   for  $m$-a.e.  $x\in \hat N$.
       Now for a measurable $E$,
        \begin{align*}
        m(f(E)\cap B_2) &=\int_{E\cap f^{-1}(B_2)} \mu_f\;dm \\
                & \le  \int_{E\cap B_1} \mu_f\;dm \\
         &=\int_{B_1} \mu_f\cdot \mathcal X_{E\cap B_1}\;dm \\
                   & \le \left(\int_{B_1} \mu_f^{1+\epsilon}\; dm\right)^{\frac{1}{1+\epsilon}} m(E \cap B_1)^{\frac{\epsilon}{1+\epsilon}}\\
                    &  \le   C m(B_1)^{-\frac{\epsilon}{1+\epsilon}}\left(\int_{B_1} \mu_f dm\right)m(E \cap B_1)^{\frac{\epsilon}{1+\epsilon}}\\
          & = C m(B_1)^{-\frac{\epsilon}{1+\epsilon}}  m(f(B_1)) m(E\cap B_1)^{\frac{\epsilon}{1+\epsilon}}.
          \end{align*}
            Hence,
             $$\frac{m(f(E)\cap B_2)}{m(f(B_1))} \le C \left(\frac{m(E\cap B_1)}{m(B_1)}\right)^{\frac{\epsilon}{1+\epsilon}}.$$
   Since $f$ is $\eta$-quasisymmetric, we have $f(B_1)\subset \eta(1) B_2$, where $\eta(1)B_2$ is the ball in $(\hat N, \hat d_{CC})$  with the same center as  $B_2$ and  with radius $\eta(1)$ times that of $B_2$.
       Now the Alhfors regularity of $\hat N$ implies
        $m(f(B_1))\le C_1 m(B_2)$ for some constant $C_1\ge 1$ depending only on  $\eta$ and the  constant in Alhfors regularity condition.
          The lemma follows.

 \end{proof}

 \begin{lemma}\label{holder}
 Let $\mathcal F$ be a compact family of $K$-quasiconformal maps  from $\hat N$ to $\hat N$.
        Then there exist positive constants $b'$, $b$, $a'$, $a$  depending only on $\mathcal F$ and $N$
    such that
    $$a' (m(E))^{b'}\le m(f(E))\le a (m(E))^b$$
     for all measurable $E\subset \hat N$ and all $f\in \mathcal F$.

 \end{lemma}

 \begin{proof}
 Since $\mathcal F$ is a compact family of  $K$-quasiconformal maps, there exists some
  homeomorphism  $\eta:[0, \infty)\ra [0, \infty)$
  such that every $f\in \mathcal F$ is  $\eta$-quasisymmetric.
    Let $r_0=1/10$.    For any $x\in \hat N$  and any $f\in \mathcal F$,   let $r''=r''(x,  f)>0$ be the   largest number  such that
     $B(f(x),  r'')\subset f(B(x, r_0))$.  Let $r'=r'(x,f)>0$  be the largest number such that
      $f(B(x, r'))\subset B(f(x), r'')$.
       By  Lemma \ref{uniformdensity}     the following holds for any  measurable $E\subset B(x, r')$:
        $$\frac{m(f(E))}{m(B(f(x), r''))}\le a\left(\frac{m(E)}{m(B(x, r_0))}\right)^b.$$

      Note that $r''$ and $r'$ are continuous functions of $x,  f$.  Since $\hat N$ and $\mathcal F$ are compact,
       we have $$r'_0=\inf_{x\in \hat N, f\in \mathcal F} r'(x,  f)>0.$$
        Also set
           $$r''_0=\sup_{x\in \hat N, f\in \mathcal F} r''(x,f)\le {\text{diam}}(\hat N).$$
         Now the space $\hat N$ can be covered by a finite number of balls  $\{B(x_j, r'_0)\}_{j=1}^k$ with
          radius $r'_0$.   Let $E\subset \hat N$ be any measurable subset.
           Then $f(E)=\cup_{j=1}^k f(E\cap B(x_j, r'_0))$.  By using the Alhfors regularity of the metric measure space
            $(\hat N, \hat d, m)$ we obtain:
            \begin{align*}
            m(f(E)) & \le \sum_{j=1}^k a\left(\frac{m(E\cap B(x_j, r'_0))}{m(B(x_j, r_0))}\right)^b m(B(f(x_j), r''_0))\\
              & \le   \sum_{j=1}^k a \frac{C {r''_0}^Q}{(r_0^Q/C)^b}m(E\cap B(x_j, r'_0))^b\\
              & \le   \frac{kaC {r''_0}^Q} {(r_0^Q/C)^b}  m(E)^b.
            \end{align*}
            This establishes the second inequality in the lemma. The first inequality  holds since the family $\mathcal F^{-1}=\{f^{-1}| f\in \mathcal F\}$ is also a compact family of $K$-quasiconformal maps.

 \end{proof}

 In the next lemma we return to working with Carnot-Caratheodory metric and Lebesgue measure on $N$.

 \begin{lemma}\label{inmeasure}
 Let $\{f_j: \hat N\ra \hat N\}$ be a sequence of  $K$-quasiconformal maps that converge uniformly to  a
     quasiconformal map $f: \hat N\ra \hat N$.  Suppose for any compact subset
   $F\subset \hat N$ satisfying $\infty,
    f^{-1}(\infty)  \notin F$, and any $\epsilon>0$,
    $|\{x\in F: K(f_j, x)\ge 1+\epsilon\}|\ra 0$ as $j\ra \infty$.
      Then
     $f$ is  conformal.

 \end{lemma}

 \begin{proof}
      We shall show that $f$ satisfies the   assumption of    Lemma   \ref{capacity1qc}.
  Let $R$ be a ring in $\hat N$ satisfying $\infty,
  f^{-1}(\infty)\notin \bar R$.
 Let $C_0$ and $C_1$ be the  two  components of $\partial R$.
 For $a>0$, let
  $$\tilde R_a=\{z\in f(R): B(z,a)\subset f(R)\}.$$
   Then     $f_j(R)\supset  \tilde R_a$  for $j$ sufficiently large (depending on $a$).

   Let  $\epsilon>0$.    Let $\tilde u_a$ be admissible for $\tilde R_a$ so that
    $$\int_{\tilde R_a} |\triangledown\tilde u_a|^Q \; dx\le C(\tilde R_a)+\epsilon.$$
      One can assume  $\tilde u_a$ is smooth and is $0$ on a neighborhood of $C_0$ and $1$ on a neighborhood of $C_1$.
        Extend $\tilde u_a$ to $\hat N$ by continuity and being constant on the two components of  $\hat N\backslash \tilde R_a$.
         Then clearly
           $$\int_{\tilde R_a} |\triangledown\tilde u_a|^Q \; dx=\int_{\tilde R}|\triangledown \tilde u_a|^Q \; dx$$
            for $\tilde R\supset \tilde R_a$.

   Let $u_j=\tilde u_a\circ f_j$.    Then $u_j$ is  admissible for $f_j^{-1}(\tilde R_a)$ and hence also admissible for $R$.
     The chain rule (see Lemma 3.7 in \cite{CC06}) implies $Du_j(x)=D \tilde u_a(f_j(x))\circ Df_j(x)$ for a.e. $x\in R$.  Hence
     $Du_j(x)|_{V_1}=D \tilde u_a(f_j(x))|_{V_1}\circ Df_j(x)|_{V_1}$
      and
       $$|Du_j(x)|_{V_1}|\le |D \tilde u_a(f_j(x))|_{V_1}|\cdot  |Df_j(x)|_{V_1}|.$$
            Notice that $|Du_j(x)|_{V_1}|=|\triangledown u_j(x)|$,
        $|D\tilde u_a(f_j(x))|_{V_1}|=  |\triangledown \tilde u_a(f_j(x))|   $  and
        $$|Df_j(x)|_{V_1}|^Q=L^Q_{f_j}(x)=l^Q_{f_j}(x) K^Q(f_j, x)\le  f'_j(x) K^Q(f_j, x)$$
           due to Lemma 3.3  and (4.1) in  \cite{CC06}.
          It follows that
   $$ |\triangledown u_j(x)|^Q\le K^Q(f_j, x) {f'_j(x)} |\triangledown \tilde u_a(f_j(x))|^Q.$$

   Set $E_j=\{x\in R:  K(f_j, x)\ge 1+\epsilon\}$.
     The assumption that $\infty,
     f^{-1}(\infty)\notin \bar R$ implies that
        there are compact subsets $K_1,  K_2\subset N$ with $\bar R\subset \mathring K_1$ and $\overline{f(R)}\subset \mathring  K_2$.
     There is a constant $C\ge 1$ such that $m(E)/C\le |E|\le C \,m(E)$
       for any $E\subset K_1 $ or $E\subset   K_2$.    By assumption $|E_j|\ra 0$ as $j\ra +\infty$.
        Now Lemma   \ref{holder}     implies $|f_j(E_j)|\ra 0$.

        We have
        \begin{align*}
        C(R)  &  \le  \int_R |\triangledown  u_j|^Q\; dx\\
        & \le \int_R K^Q(f_j, x) {f'_j}(x) |\triangledown \tilde u_a(f_j(x))|^Q\; dx\\
        & \le (1+\epsilon)^Q\int_{f_j(R)} |\triangledown \tilde u_a|^Q\; dx +K^Q\int_{f_j(E_j)}|\triangledown \tilde u_a|^Q\; dx.
        \end{align*}
         Since  $|f_j(E_j)|\ra 0$ as $j\ra +\infty$,   the second term above goes to $0$ as $j\ra \infty$.
           Hence
          $$C(R)\le (1+\epsilon)^Q(C(\tilde R_a)+\epsilon),$$
          and so $C(R)\le C(\tilde R_a)$ for small $a$.  We claim that $C(\tilde R_a)\ra C(f(R))$ as $a\ra 0$.
           Indeed,  for any $\delta>0$    there is a  smooth function $v$,  admissible for  $f(R)$ with value 0 in a neighborhood of $f(C_0)$  and value 1 in a neighborhood of $f(C_1)$,  such that
            $$\int_{f(R)} |\triangledown v|^Q\; dx\le C(f(R))+\delta.$$
             For small enough $a$, $v$ is also admissible for  $\tilde R_a$ and
              $$\int_{f(R)} |\triangledown v|^Q\; dx=\int_{\tilde R_a} |\triangledown v|^Q\; dx.$$
               Since
               $$\int_{\tilde R_a} |\triangledown v|^Q\; dx\ge C(\tilde R_a)\ge C(f(R))\ge         \int_{f(R)} |\triangledown v|^Q\; dx-\delta$$
                and $\delta$ is arbitrary,
            we have     $C(\tilde R_a)\ra C(f(R))$ as $a\ra 0$. Hence                $C(R)\le C(f(R))$.

                Finally  as $K(f_j^{-1}, f_j(x))=K(f_j, x)$,   Lemma \ref{holder} implies that $f_j^{-1}$, $f^{-1}$ also satisfy the assumption in the lemma and   we conclude that  $C(f(R))\le C(R)$.
                   Now we have verified the assumption   of    Lemma \ref{capacity1qc}
                     and so $f$ is conformal.

 \end{proof}

 \section{Tukia-type  theorem for Carnot groups}

In this section we prove Theorem \ref{tukia} and Corollary \ref{quasiaction}.

\subsection{Existence of invariant measurable conformal structure}\label{conformalstructure}

In this subsection we  show that every
   uniform quasiconformal group of $\hat N$ leaves invariant a measurable conformal structure on
      $\hat N$.

  Let  $m\ge 2$  be an integer and   $X$     the space of symmetric, positive definite   real  $m\times m$ matrices with determinant $1$.
 The  general linear  group $GL(m, \mathbb R)$ acts on $X$ by
  $$M[S]=  (\det M)^{-\frac{2}{m}}MSM^T$$
      for $M\in GL(m,\mathbb R)$ and $S\in X$,   where $M^T$ is the transpose of $M$.
   Then $SL(m, \mathbb R)\subset GL(m,\mathbb R)$
        acts on $X$ transitively,    and the
    stabilizer of  $SL(m, \mathbb R)$ at  $I_m$ is $SO(m)$.  Hence  $X=SL(m,  \mathbb R)/SO(m)$.
         The homogeneous manifold  $X=SL(m,  \mathbb R)/SO(m)$
          is a symmetric space of non-compact type  (see table V on page 518 of \cite{Hel78})   and so has nonpositive sectional curvature.

          The Riemannian distance $d$ on $X$ is invariant under the action of $GL(m, \mathbb R)$  and satisfies:
            $$d(I_m, ODO^T)=\sqrt{a_1^2+\cdots +a_m^2},$$
            where $I_m$  is the identity matrix, $O$ is orthogonal and    $D$ is diagonal with diagonal entries  $e^{a_1}\ge \cdots \ge e^{a_m}$,
             see \cite{Ma}, p.27.   We notice that the  dilatation of  a non-singular  linear map $A:=UDV^T: \mathbb R^m\ra \mathbb R^m$ (where $D$ is as above and $U,V$ are orthogonal)
             is given by:
                  $$ K(A):=K(A, 0)=\frac{\max\{|AX|: |X|=1\}}{\min\{|AX|: |X|=1\}}=\frac{e^{a_1}}{e^{a_m}}=e^{a_1-a_m}.   $$
              It follows that there is a function $\phi: [0, \infty)\ra [0, \infty)$  (one may choose $\phi(t)=e^{t}-1$) with $\phi(t)\ra 0$ as $t\ra 0$ such that
               \begin{equation}\label{K(A)}
               K(A)\le 1+ \phi(d(I_m, AA^T)).\end{equation}

Let $N$ be a Carnot group with  Lie algebra $\mathfrak n=V_1\oplus \cdots \oplus V_r$.
Let $V_1$ be equipped with an  inner product.  Let $m=\dim(V_1)$.  By using an orthonormal basis for $V_1$,
 we can identify   the special orthogonal group  $SO(V_1)$  associated with the inner product with $SO(m, \mathbb R)$ and similarly identify $SL(V_1)$ with $SL(m, \mathbb R)$.
     Hence
      $$X=SL(V_1)/O(V_1).$$

 A   measurable conformal structure on $\hat N=N\cup\{\infty\}$ is an essentially bounded
  measurable map
 $$\mu: U \ra     X$$
  defined on a full measure subset $U\subset N$.     Note we do not require $\mu$ to be defined everywhere on $\hat N$. In particular, we prefer $\mu$ not defined at $\infty$.

 For a  quasiconformal map $f: \hat N\ra \hat N$ and a measurable conformal structure $\mu$ on $\hat  N$, the pull-back
  measurable
  conformal structure
  $f^*\mu$ is defined  by:
   $$f^*\mu (x)=(Df(x)|_{V_1})^T[\mu(f(x))], \;\;\;\text{for a.e.}\; x\in N.$$
    Here we are using the fact that a quasiconformal map is Pansu differentiable a.e.
      This is similar to the definition of pull-back  of a Riemannian metric under a smooth map.
     It is easy to check that
       $(g\circ f)^*\mu=f^*(g^*\mu)$,   where  $f, g : \hat N\ra \hat N$ are quasiconformal maps   and $\mu$ is a
        measurable conformal structure on $\hat N$.

         A quasiconformal map $f$ is called conformal with respect to the measurable  conformal structure $\mu$
            if $f^*\mu=\mu$;   in other words,   $\mu(x)=(Df(x)|_{V_1})^T[\mu(f(x))]$
              for a.e. $x\in N$.

     \begin{Prop} \label{existence}
      Let $G$ be a   
       uniform quasiconformal   group
     of $\hat  N$.   Then there is a      measurable conformal structure on $\hat N$
       such that every $g\in G$ is conformal with respect to $\mu$.

     \end{Prop}

 \begin{proof}   We first  assume $G$ is countable.
  Since $G$ is countable and for each $g\in G$,  the Pansu differential $Dg(x)$ exists and is a graded automorphism for  a.e. $x\in N$,
   there is      a measurable, $G$-invariant subset $U$ of full measure such that   $Dg(x)$ exists and is a graded automorphism
        for all $g\in G$ and  at all $x\in U$.
     Let $\mu_0$ be the  left invariant conformal   structure  on $N$ associated with an inner product  on $V_1$
       (so $\mu_0: N\ra X$ is a constant map).
       For each $x\in U$, set
        $M_x=\{g^*\mu_0(x): g\in G$\}.
          Since $G$ is a uniform quasiconformal group,  $M_x$ is a bounded subset of $X$.
        The assignment $x \mapsto M_x$ is $G$-invariant:   for any $f\in G$,
         $$f^* M_{f(x)}=f^*\{g^*\mu_0(f(x)): g\in G\}=\{{(g\circ f)}^*\mu_0(x):  g\in G\}=M_x.$$
           Since $M_x$ is a  bounded subset of the non-positively curved symmetric space
              $X$,  there exists a unique circumcenter  $P(M_x)$  in  $X$  for the subset $M_x$, see for example  p.179 of  \cite{BH}.
                  Define $\mu:  U\ra X$ by  $\mu(x)=P(M_x)$.      Since the assignment $x\mapsto M_x$ is $G$-invariant,
                  $\mu$ is also $G$-invariant.

                  It remains to show $\mu$ is measurable.   Enumerate $G=\{g_0, g_1, \cdots\}$, and let
                   $M(x,j)=\{{g_i^*\mu_0}: i\le j\}$ and $\mu_j(x)=P(M(x, j))$.    Now, $Y\mapsto P(Y)$ is   continuous with respect to the
                   Hausdorff metric (see p.334 of \cite{T86}).    This implies  $\mu_j$ is measurable.  Since $\mu_j$ converges to $\mu$ point-wise,  $\mu$ is  also measurable.
 This completes the proof when $G$ is countable.

 Now let $G$ be a  general   uniform quasiconformal group of $\hat N$.  Let $G_0\subset G$ be a countable subgroup of $G$ that is dense in the topology of uniform convergence. By the above argument, there is a   $G_0$-invariant   measurable conformal structure $\mu$.    It follows that $\mu$ is also  $G$-invariant   as a  uniform limit of a sequence of $\mu$-conformal maps is $\mu$-conformal.   This  follows from  an analogue of Theorem D in \cite{T86}  in the setting of Carnot groups.   The proof of Theorem D in \cite{T86} is valid for Carnot groups   after some small modifications: to blow up maps in the Carnot group setting one needs to conjugate using Carnot dilations.

 \end{proof}

 Let $G_0$ be a countable dense subgroup of $G$ as above.  If the induced action of $G$ on the space of distinct triples of $\hat N$ is cocompact, then the same is true for the induced action of $G_0$.  Another way to handle the case of an uncountable group $G$ is to first run the argument from the next subsection to conjugate $G_0$ into the conformal group, then the same  map also conjugates $G$ into the conformal group   since the limits of conformal maps are conformal maps.

 \subsection{Conjugating into the conformal group}\label{conjugate}

In this subsection we will prove Theorem \ref{tukia}    and Corollary \ref{quasiaction}.

A map $\mu: U\ra Y$, where $U\subset N$ is open and $Y$ is a metric space, is called approximately continuous at $x_0\in U$
   if for any $\epsilon>0$,   the set  $\mu^{-1}(B(\mu(x_0), \epsilon))$ has density 1 at $x_0$; that is, if
    $$\frac{|B(x_0, r)\cap \mu^{-1}(B(\mu(x_0), \epsilon))|}{|B(x_0, r)|}\ra 1$$
     as $r\ra 0$.        By  Theorem 2.9.13 in \cite{F69},    if $Y$ is separable and $\mu$ is measurable, then $\mu$ is approximately continuous a.e.

     Let $N$ be  a Carnot group and $S=N\rtimes \R$ the negatively curved homogeneous manifold associated with
      $N$.
          There is a natural map    $\pi: T(\partial S) \ra S$,  that assigns to each  distinct triple
     $(\xi_1, \xi_2, \xi_3)\in T(\partial S)$ the center of  the triple.  To be more precise,
          $\pi(\xi_1, \xi_2, \xi_3)$ is defined to be the orthogonal projection of $\xi_3$ onto the complete geodesic
           $\xi_1\xi_2$.       We observe that for any compact $C\subset S$, the set $\pi^{-1}(C)$ is compact in
            $T(\partial S)$.

       Let $G$ be a group of homeomorphisms of $\partial S$.    Then $G$ also acts diagonally on $T(\partial S)$:
         $g(\xi_1, \xi_2,\xi_3)=(g(\xi_1), g(\xi_2), g(\xi_3))$.       Recall that a point $\xi\in \partial S$ is said to be  a  radial limit point of
            $G$ if   there  exists
            a sequence of elements $\{h_j\}_{j=1}^\infty$ of $G$  with the following property:
                           for any  triple $T=(\xi_1, \xi_2, \xi_3)\in
             T(\partial S)$,
             and       any   complete  geodesic $\gamma$    asymptotic to $\xi$,
              there exists a constant $C>0$          with           $\pi(h_j(T))\ra \xi$  and   $d(\pi(h_j(T)), \gamma)\le C$.

     We  recall  that each inner product on $V_1$ determines   a left invariant Carnot-Caratheodory metric on $N$.

     \begin{theorem}\label{radial}
     Let $G$ be a uniform quasiconformal group of $\hat N$  and $\mu$ a $G$-invariant  measurable conformal
     structure on $\hat N$.    Suppose there is a point $p\in N\subset \hat N$ such that $\mu$ is approximately continuous
      at $p$ and $p$ is also a radial limit point for $G$.     Then there exists a quasiconformal map $f:\hat N\ra \hat N$
        and an inner product on $V_1$    such that $fGf^{-1}$ consists of conformal maps   with respect to
           the left invariant Carnot-Caratheodory metric determined by this inner product.

     \end{theorem}

     \begin{proof}
       The left translation action of $S=N\rtimes \R$ on itself is by isometries.
       The  elements in $\R$ translate the vertical geodesic above the origin $0$ in $N$  and
          the boundary homeomorphisms induced  by them are the standard Carnot group dilations of $N$. We shall
            use $\tilde \delta_t$ to denote  the isometry of $S$ that induces the standard Carnot group dilation $\delta_t$.

      We may assume $p=0$ is the origin of $N$.
              Fix a triple $T\in T(\partial S)$  and
               let $\gamma$ be the vertical geodesic above $0$.
              Since $0$ is a radial limit point of $G$, there exists
            a sequence of elements $\{h_j\}_{j=1}^\infty$ of $G$    and
              a constant $C>0$          with           $\pi(h_j(T))\ra 0$  and   $d(\pi(h_j(T)), \gamma)\le C$.
               Fix a point $x_0\in \gamma$.
              For each $j$ there is some $s_j>0$ with $s_j\ra +\infty$ as $j\ra \infty$ such that
               $d(\tilde \delta_{s_j}\circ \pi\circ h_j(T), x_0)\le C$.
                Since $\tilde \delta_{s_j}$ is an isometry of $S$, we have $\tilde \delta_{s_j}\circ \pi=\pi\circ \delta_{s_j}$.
                  Hence $d(\pi\circ \delta_{s_j}\circ h_j(T), x_0)\le C$  and  so the set $\{\delta_{s_j}\circ h_j(T)\}_{j=1}^\infty$ lies in the compact subset
                   $\pi^{-1}\bar B(x_0, C)$.    It follows that the family   $\{\delta_{s_j}\circ h_j\}_{j=1}^\infty$  of
                    $K$-quasiconformal maps  is precompact.
       Define $f_j: \hat N\ra \hat N$ by $f_j=\delta_{s_j}\circ h_j$.
        By passing to  a  subsequence, we may assume $f_j$ converges uniformly to a $K$-quasiconformal map
        $f: \hat N\ra \hat N$.
      We shall show  that for any $g\in G$,  $fgf^{-1}$ is  conformal with respect to the left invariant Carnot-Caratheodory metric determined by $\mu(0)$.

     Let  $g\in G$ and denote $\tilde g=fgf^{-1}$, $\tilde g_j=f_jgf_j^{-1}$.   Let  $\mu_j=(f_j^{-1})^*\mu$.
                    Since $\mu$ is $G$-invariant,      $\tilde g_j$ is conformal with respect to $\mu_j$:
                       $$\tilde g_j^*\mu_j=(f_j^{-1})^*g^*f_j^* (f_j^{-1})^*\mu=(f_j^{-1})^*g^*\mu=(f_j^{-1})^*\mu=\mu_j.$$
                         Note  $\mu_j=(f_j^{-1})^*\mu=(\delta_{s_j}^{-1})^* (h_j^{-1})^*\mu=(\delta_{s_j}^{-1})^*\mu$.
                          So for $q\in N$,
                          $$\mu_j(q)=(\delta_{s_j}^{-1})^*\mu (q)=(D\delta_{s_j}^{-1}(q)|_{V_1})^T[\mu(\delta_{s_j}^{-1}(q))]
                          =\mu(\delta_{s_j}^{-1}(q))$$
            since  $D\delta_t(q)|_{V_1}$ is an Euclidean dilation.

            Let $F\subset \hat N$ be  a compact subset such that  $\infty,
            \tilde g^{-1}(\infty)
            \notin F$.  There is a compact  subset $F_0$  of $N$ such that $F\subset F_0$,
            $ \tilde{g}_j(F)\subset F_0$ for all sufficiently large $j$.
              Since $\mu$ is approximately continuous at $0$  and $s_j\ra \infty$, the equality $\mu_j(q)=\mu(\delta^{-1}_{s_j}(q))$
               implies that    for any $\epsilon>0$    there are subsets $A_j\subset F_0$ with $|A_j|\ra 0$ as $j\ra \infty$
                  and $d(\mu_j(x), \mu(0))\le \epsilon$ for $x\in F_0\backslash A_j$.

                   The maps $\tilde g^{-1}_j$ and $\tilde g^{-1}$ form a compact family of $K$-quasiconformal maps. By Lemma
                    \ref{holder},  we have $|\tilde g_j^{-1}(A_j)|\ra 0$ as $j\ra \infty$.
                            Set $B_j=A_j\cup \tilde g_j^{-1}(A_j)$.
                      Now we have $|B_j|\ra 0$ as $j\ra \infty$ and
                       $d(\mu_j(x), \mu(0))\le \epsilon$ and $d(\mu_j(\tilde g_j(x)), \mu(0))\le \epsilon$ for
                        $x\in F\backslash B_j$.
     Since  $\tilde g_j$ is $\mu_j$-conformal, we have
      $\mu_j(x)=(D\tilde g_j(x)|_{V_1})^T [\mu_j(\tilde g_j(x))]$ for a.e. $x$.
      Now
   $$    d(\mu_j(x), (D\tilde g_j(x)|_{V_1})^T[\mu(0)])
   =    d(\mu_j(\tilde g_j(x)), \mu(0))\le \epsilon    $$
    for a.e. $x\in F\backslash B_j$.         Combining this with  $d(\mu_j(x), \mu(0))\le \epsilon$  we   get
       $d(\mu(0),  (D\tilde g_j(x)|_{V_1})^T[\mu(0)])\le 2\epsilon$
     for        a.e.  $x\in F\backslash B_j$.
       By
            Lemma 3.3 in \cite{CC06}  and
       (\ref{K(A)}) we have $K(\tilde g_j, x)= K(D\tilde g_j(x)|_{V_1})\le 1+ \phi(2\epsilon)$  for
         a.e.  $x\in F\backslash B_j$.
    This implies the assumption of   Lemma \ref{inmeasure}
 is satisfied and so $\tilde g$ is conformal.

     \end{proof}

\noindent
{\bf{Proof of Theorem \ref{tukia}.}}
  Let $G_0$ be a countable dense subgroup of $G$.  
   By  Proposition \ref{existence},   there exists a $G_0$-invariant measurable conformal structure
   $\mu$ on $\hat N$.     Since the    action   of $G_0$ on $T(\partial S)$ is   co-compact,
        every point in $\partial S$ is a radial limit point.
    By    Theorem 2.9.13 in \cite{F69},    $\mu$ is  approximately continuous at  a.e. point in $\hat N$.   By
     Theorem \ref{radial},       there exists a quasiconformal map
     $f$ of $\hat N$ such that    $fG_0f^{-1}$     is a conformal group  of $\hat N$ with respect to   some  left invariant Carnot-Caratheodory metric    $d_{CC}$  on $N$.
Since  $fG_0f^{-1}$    is dense in  $fGf^{-1}$    and the limits   of conformal maps are  conformal,  we  conclude that $fGf^{-1}$  is also a conformal group   with respect to 
 $d_{CC}$.

     For any  left invariant Carnot-Caratheodory metric $d$   on $N$, let $\text{Conf}(\hat N, d)$ be the group of conformal maps of $\hat N$ with respect to $d$.
     We next show that there is a fixed  left invariant Carnot-Caratheodory metric $d_0$   on $N$  such that
   for any   left invariant Carnot-Caratheodory metric      $d$ on $N$,  the group $\text{Conf}(\hat N, d)$
     can be  conjugated into $\text{Conf}(\hat N, d_0)$ by a graded automorphism of $N$.
      For this we use   the result of Cowling-Ottazzi  \cite{CO15}   on conformal maps of Carnot groups.  By Theorem 4.1 of   \cite{CO15} there are two cases
  depending on whether $N$ is the Iwasawa $N$ group of a real  rank-one simple Lie group.

    Let $d$ be a  left invariant Carnot-Caratheodory metric  on $N$.
    First assume $N$  is the Iwasawa $N$ group of a real  rank-one simple Lie group
    $I$. In this case, Theorem 4.1 of   \cite{CO15}   states  that  the action of each $g\in \text{Conf}(\hat N, d)$
     on $\hat N$ agrees with the action of some $\phi(g)\in I$. This $\phi(g)$ is clearly unique and the map $\text{Conf}(\hat N, d)\ra I, g\mapsto \phi(g)$ defines an injective homomorphism.
       It is known that in this case there is a  left invariant Carnot-Caratheodory metric $d_0$   on $N=\partial S\backslash \{\infty\}$  such that $I=\text{Conf}(\hat N, d_0)$.
        Hence the statement holds in this case.

     Next we assume $N$  is   not the Iwasawa $N$ group of a real rank-one simple Lie group. In this case, Theorem 4.1 of   \cite{CO15}   states that for each $g\in \text{Conf}(\hat N, d)$, we have $g(N)=N$ and $g|_N: (N, d)\ra (N, d)$   is a similarity.  In other words, we have  $\text{Conf}(\hat N, d)=\text{Sim}(N,d)$.  Now we finish the proof by applying Lemma \ref{maximal}.

\qed

\noindent
{\bf{Proof of Corollary \ref{quasiaction}}}.
 Let $N$ be a Carnot group and $S=N\rtimes \mathbb R$ be the associated
    solvable Lie group.
   Let $G$ be a  
    group that quasi-acts  co-boundedly on  $S$.
This induces a uniform quasiconformal action of $G$ on $\partial S=\hat N$ such that the induced action on the space of distinct triples is co-compact. By Theorem
 \ref{tukia}   there is  a  fixed (independent of $G$)    left invariant Carnot-Caratheodory metric  $d_0$ on $N$ such that     $G$
   is quasiconformally conjugate to  a subgroup of   $\text{Conf}(\hat N, d_0)$.
    It  now suffices to show that there is a left invariant Riemannian metric  $g_0$ on $S$
     such that   every map in   $\text{Conf}(\hat N, d_0)$   is the boundary map of some isometry of $(S, g_0)$.     Again we will use  Theorem 4.1 of   \cite{CO15}.

    First assume $N$  is the Iwasawa $N$ group of a real  rank-one simple Lie group
    $I$. In this case,   as observed in the proof of Theorem \ref{tukia}, $\text{Conf}(\hat N, d_0)$ injects into $I$,
     which is the isometry group of a left invariant
         Riemannian metric   $g_0$ on $S$.
      Next we assume $N$  is   not the Iwasawa $N$ group of a real  rank-one simple Lie group. In this case,  each $g\in \text{Conf}(\hat N, d_0)$ has the form
        $g|_N=L_{n}\circ \delta_{t_g}\circ A_g$  for some $n\in N$, $t_g>0$, where  $L_n$ is left translation by $n\in N$, $\delta_t$ ($t>0$) is  a Carnot dilation, and
       $A_g:(N,  d_{0})\ra (N, d_{0})$ is an  isometry and also  a graded automorphism of $N$.  Let $\mathfrak n=V_1\oplus \cdots\oplus V_k$ be the Carnot grading of
        $\mathfrak n$.  Then the map $\text{Conf}(\hat N, d_0)\ra GL(V_j), g\mapsto A_g|_{V_j}$ is a group homomorphism whose image has  compact closure.  It follows that there is some    $\text{Conf}(\hat N, d_0)$-invariant    inner product $<,>_j$ on $V_j$.
         Now we  equip $\mathfrak s=\mathfrak n \rtimes \mathbb R$ with the inner product  $<,>$ satisfying: (1) $<,>$ agrees with $<,>_j$ on $V_j$; (2) the subspaces $V_j$, $1\le j\le k$  and $\{0\}\times \mathbb R\subset  \mathfrak n \rtimes \mathbb R$ are all perpendicular to each other with respect to $<,>$.
 For each $g\in \text{Conf}(\hat N, d_0)$, define a map $\phi(g): S\ra S$ by $\phi(g)(x,t)= L_{(n,t_g)}(A_g x, t)$, where $ L_{(n,t_g)}$ is the left translation of $S$ by $(n,t_g)\in S$.   Then $\phi(g)$ is an isometry   of  $(S, g_0)$   with its induced boundary map equal to $g|_N=L_{n}\circ \delta_{t_g}\circ A_g$
  and $\phi: G\ra \text{Isom}(S, g_0)$ is a group homomorphism, where $g_0$ is the   left invariant Riemannian metric on $S$ determined by $<,>$.

Finally we recall that a self quasiconformal map of $\partial S=\hat N$ extends to a quasi-isometry of $S$. Combining this  with the preceding two paragraphs we conclude that the original quasi-action of $G$ on $S$ is quasi-conjugate to an isometric action of $G$  on $(S, g_0)$.

\qed

\section{A fibered  Tukia theorem for diagonal Heintze pairs}

In this section we prove a  fiber bundle  version of Tukia's Theorem (Theorem \ref{foliatedtheorem})
   in the spirit of \cite{D10} for  diagonal  Heintze pairs. 
   We first explain that such a group $N$ admits an iterated fibration  structure  with Carnot group fibers and that each self biLipschitz map of $N$ induces
     bundle
   maps.
   The  fiber Tukia theorem  states  that after a  biLipschitz  conjugation  the induced maps between the fibers are similarities.

\subsection{BiLipschitz maps of  a diagonal Heintze pair
}\label{bilip-diagonal}

Before we prove the fiber  Tukia theorem, we first look at  individual biLipschitz maps of   a diagonal Heintze pair $(N, D)$.

 Let $(N, D)$ be a diagonal Heintze pair  and $0<\lambda_1<\cdots <\lambda_r$ the distinct eigenvalues of $D$.
  Let $\mathfrak n=\oplus_j V_{\lambda_j}$ be the decomposition of the Lie algebra $\mathfrak{n}$ of $N$ into eigenspaces.
 If  $V\subset \mathfrak n$ is a   linear subspace such that $D(V)\subset V$ then $V$ is graded, that is,
  $V=\oplus_j (V\cap V_{\lambda_j})$.  It follows that if $V\subset V'$ are two linear subspaces of $\mathfrak n$ such that $D(V)\subset V$,
  $D(V')\subset V'$,   then $V'/V=\oplus_j (V'\cap V_{\lambda_j})/(V\cap V_{\lambda_j})$  and $D$ induces a linear map on $V'/V$ and acts on
    $ (V'\cap V_{\lambda_j})/(V\cap V_{\lambda_j})$   by multiplication by $\lambda_j$.

 Now let  $\mathfrak h_1$ be the Lie sub-algebra of $\mathfrak n$  generated by $V_{\lambda_1}$.
   We say $(N, D)$  is of Carnot type if $\mathfrak h_1=\mathfrak n$.
        In general, inductively define  $\mathfrak h_i=N(\mathfrak h_{i-1})$ for $i\ge 2$,
   where for a  Lie sub-algebra $\mathfrak h\subset \mathfrak n$,  $N(\mathfrak h)=\{X\in \mathfrak n| [X, Y]\in \mathfrak h, \forall Y\in \mathfrak h\}$  denotes
      the normalizer of $\mathfrak h$ in $\mathfrak n$.
   Then there is some integer 
    $m\ge 1$ such that   $\mathfrak h_m=\mathfrak n$.   We next explain how to refine the sequence $0< \mathfrak h_1 < \cdots <  \mathfrak h_m$. 
     It is not hard to see that $D(\mathfrak h_1)\subset \mathfrak h_1$, and by using the definition of  derivation we obtain
  $D(\mathfrak h_i)\subset \mathfrak h_i$ for all $i$.   Hence   $D$ induces a  diagonal derivation  $\bar D$ of  
   $\mathfrak h_i/{\mathfrak h_{i-1}}$   with positive eigenvalues.
 Then for every $i$, if   $(\mathfrak h_i/{\mathfrak h_{i-1}}, \overline D)$ is of
        non-Carnot type,  we can repeat the above process: first take the Lie sub-algebra generated  by the eigenspace of  the smallest eigenvalue of $\overline D$, then take the normalizers.  In this way we obtain  a sequence of Lie sub-algebras of 
        ${\mathfrak h_i}/{\mathfrak h_{i-1}}$:    $0=\bar{\mathfrak h}_{i,0}< \bar{\mathfrak h}_{i,1}< \cdots <  \bar{\mathfrak h}_{i,m_i}={\mathfrak h_i}/{\mathfrak h_{i-1}}$.   Let $q_i: \mathfrak h_i\ra {\mathfrak h_i}/{\mathfrak h_{i-1}}$ be the quotient map and set 
         $\mathfrak h_{i,j}=q_i^{-1}(\bar{\mathfrak h}_{i,j})$.    The refinement of the sequence    $0< \mathfrak h_1 < \cdots \mathfrak <  \mathfrak h_m$ is obtained   as follows:  for those  $i$ such that $(\mathfrak h_i/{\mathfrak h_{i-1}}, \overline D)$ is of
        non-Carnot type, 
         we  insert    between  $\mathfrak h_{i-1}$ and $\mathfrak h_i$ the sequence 
          $\mathfrak h_{i,1}< \cdots < \mathfrak h_{i, m_i-1}$.   
        This refinement process can be further repeated and eventually we must stop since $\mathfrak n$ is finite dimensional.
             At the end, we obtain a sequence of    Lie sub-algebras
          $0=\mathfrak n_0<\mathfrak n_1<\cdots <\mathfrak n_s=\mathfrak n$    such  that $\mathfrak n_{i-1}$ is an ideal of $\mathfrak n_{i}$   for every $i$, $D(\mathfrak n_i)\subset \mathfrak n_i$,     and  each
              $(\mathfrak n_{i}/{\mathfrak n_{i-1}},  \bar D)$ is of Carnot type (that is,
               $\mathfrak n_{i}/{\mathfrak n_{i-1}}$
             is a  Carnot  algebra and  $D$ induces a derivation $\bar D$   of   $\mathfrak n_{i}/{\mathfrak n_{i-1}}$   that is a multiple of a Carnot derivation).   We remark that when we talk about Carnot algebra here we include the case of abelian Lie algebra.  In
              other words,
               $(\mathfrak n_i/{\mathfrak n_{i-1}}, \overline D)$ is of
       Carnot type
                when
               $ \mathfrak n_{i}/{\mathfrak n_{i-1}}$ is  some $\mathbb R^n$ and $\bar D$ is  a  standard Euclidean dilation.
 Let $N_i$ be the connected Lie subgroup of $N$ with Lie algebra $\mathfrak n_i$. Then
   $N_{i-1}$ is normal in $N_i$ and  $N_{i}/{N_{i-1}}$ is a Carnot group.

 \begin{definition}\label{defn:preserved}Given     a
  diagonal Heintze pair
  $(N,D)$, we call
$1=N_0<N_1<\cdots <N_s=N$  the sequence of  subgroups defined by the process above the \emph{preserved subgroups  sequence}.
\end{definition}

 Let $d$ be a $D$-homogeneous distance on $N$.  The restriction of $d$ on $N_i$ is  a $D|_{\mathfrak n_i}$-homogeneous distance on $N_i$. As $N_{i-1}\triangleleft N_i$,
  $d$ induces a distance  $\bar d$   on $N_i/{N_{i-1}}$, see end of Section \ref{homo distance}.      By the preceding paragraph, this distance $\bar d$ is a $\bar D$-homogeneous distance    and hence is biLipschitz with
   $\bar d_{CC}^{\frac{1}{\lambda}}$, where $\bar d_{CC}$ is a Carnot-Caratheodory metric on $N_i/{N_{i-1}}$ and $\lambda>0$ is the smallest eigenvalue of $\bar D$.
   For each $i$ we have a fibration $\pi_i: N/{N_{i-1}}\to N/{N_i}$ with fiber $N_i/{N_{i-1}}$.     The distance $d$ does not induce  any distance on $N/{N_i}$ when $N_i$ is not normal in $N$.
      However, as indicated above, $d$ still induces a distance $\bar d$ on the fibers $N_i/{N_{i-1}}$.

      The following result in particular applies to biLipschitz maps of $N$ (when $\mathfrak h_1\not=\mathfrak n$).

      \begin{theorem}(\cite{CP17}) \label{permutes}
 Suppose $(N, D)$ is not of Carnot type. Then every  quasisymmetric map
 $F: N\rightarrow N$  permutes the left cosets of $N_1$  and is biLipschitz.
 \end{theorem}

 By  Lemma 3.2, \cite{CPS17} (see also Lemma 4.4, \cite{LX15} in the Carnot case),   for a   connected  Lie subgroup $H\subset N$, two left cosets $g_1H$, $g_2H$ are  at finite Hausdorff distance from each other  if and only if $g_1H$, $g_2H$    lie in the same
    left coset of $N(H)$, the normalizer of $H$ in $N$.
  It follows that a biLipschitz map $F:  N\ra N$ permutes the cosets of $N_i$ for every $i$.
      As a consequence,  $F$ induces a map
  $F_i:  N/{N_i}\ra  N/{N_i}   $   and a bundle  map  of the fibration
   $\pi_i:    N/{N_{i-1}}\to N/{N_i}$.   In general we cannot talk about the metric property of these maps as $d$ does not induce a metric on
    $N/{N_i}$   when $N_i$ is not normal in $N$.  However,
   the  restriction  of $F$ to cosets of $N_i$   yields a biLipschitz map of $N_i$,   that is, $F_p|_{N_i}=(L_{F(p)^{-1}}\circ F\circ L_p)|_{N_i}: N_i\to N_i$  is biLipschitz for every $p\in N$.    As $F_p|_{N_i}$ also permutes the cosets of $N_{i-1}$, it      induces a
       biLipschitz map $F_{i,p}:   N_i/{N_{i-1}}  \ra  N_i/{N_{i-1}}$.
         In other words,  the restrictions of $F_{i-1}$  to  the fibers  of the fibration
          $\pi_i:   N/{N_{i-1}}\to N/{N_i}$  are biLipschitz maps between the fibers.
   Observe  that as $N_{i-1}\triangleleft N_i$,  we have $F_{i,p}=F_{i,q}$ when $q=pw$ for some $w\in N_{i-1}$.  Because of this,   we may define  $F_{i,h}$ for $h\in N/{N_{i-1}}$ by $F_{i,h}=F_{i,p}$ for any $p\in N$  satisfying $h=pN_{i-1}$.

\begin{remark}\label{samePder} Note that for  $q=pw$ with $w\in N_i-N_{i-1}$,
  we have
 $F_{i,q}(N_{i-1}) \neq F_{i, p}(wN_{i-1})$ however  the Pansu differentials satisfy
$DF_{i,p}(wN_{i-1})=DF_{i, q}(0)$ by the definition of Pansu   differential.
\end{remark}

For convenience, we introduce the following terminology.

 \begin{definition}\label{def of post foliated}
  Let $(N, D) $ be a  diagonal Heintze pair   with
  preserved subgroups  sequence
$1=N_0<N_1<\cdots <N_s=N$.
   Let $F: N\ra N$ be a
   biLipschitz map.  
    We say  $F$
   is   an    $i$-similarity  for  some $1\le i\le s$  if
 there exists  a Carnot-Caratheodory metric $\bar d_i$ on $N_i/{N_{i-1}}$   such that
 $F_{i,p}:   (N_i/{N_{i-1}}, \bar d_i)   \ra  (N_i/{N_{i-1}}, \bar d_i)$
     is a similarity  for any $p\in N$.
   The map  $F$ is  a  fiber  similarity  if it is  an  $i$-similarity for each $i$.
 A uniform quasisimilarity group   $\Gamma$ of  $N$ is a fiber  similarity group
    if
                          there exist  Carnot-Caratheodory metrics $\bar d_i$ on $N_i/{N_{i-1}}$   such that
                        each $\gamma\in \Gamma$    is  a fiber  similarity
                          with respect to these metrics.

\end{definition}

\subsection{Invariant measures on homogeneous spaces}

Let $N$ be  a simply connected nilpotent Lie group and $H$ a closed connected Lie subgroup.   Then the homogeneous space $N/H$ admits a $N$-invariant measure.
 By  Theorem 1.2.12  and its proof  in \cite{CG90},  the   product of a Haar  measure $m_H$ on $H$ and an invariant measure $m_{N/H}$ on $N/H$
   is a  Haar measure on $N$,
 in the following sense.  There is a smooth submanifold $K\subset N$ and a measure $m'$ on $K$  with the following properties:
   (1) the map $P: K\times H\ra N$, $(k,h)\mapsto kh$, is a diffeomorphism and $P_*(m'\times m_H)$ is a  Haar measure on $N$; (2) $(\pi\circ \iota)_*(m')=m_{N/H}$, where $\iota: K\subset N$ is the inclusion and $\pi: N\ra N/H$ is the  quotient map.

   \begin{lemma}\label{invariantm}
       Let $(N,D)$ be a diagonal  Heintze pair, and $H$ a  connected Lie subgroup of $N$.
          Let $F: N\ra N$ be a biLipschitz map  that permutes the cosets of $H$.   Let $\bar F: N/H\ra N/H$ be the induced map of the homogeneous space $N/H$.   Then  there is a constant $C$ depending only on the biLipschitz constant of $F$ such that
           for any  measurable subset $A\subset N/H$ the inequality
              $\frac{1}{C} m_{N/H}(A)\le m_{N/H}(\bar F(A))\le C m_{N/H}(A)$ holds, where
               $m_{N/H}$ is an invariant measure on  $N/H$.

   \end{lemma}

   \begin{proof}
   Since the Hausdorff measure on $N$ with respect to a $D$-homogeneous distance  is a Haar measure and $F$ is biLipschitz,  there is a constant depending only on the biLipschitz constant of $F$ such that
   $\frac{1}{C} m_N(U)\le m_N(F(U))\le C m_N(U)$  for any measurable subset $U\subset N$.   Similarly, since the restriction of $F$ to cosets of $H$  is also biLipschitz, we have
    $\frac{1}{C} m_H(B)\le m_H(F(B))\le C m_H(B)$  for measurable subsets  $B$  of  cosets of $H$.    Now let $A\subset N/H$ be a measurable subset.   Let $A'=\pi^{-1}(A)\cap K$, where $K$ is   as above.  Let $B$ be a  bounded open subset of $H$.
    Set $U=P(A'\times B)\subset N$.    Then    $m_N(U)=m_H(B) m_{N/H}(A)$.
    By the above we have
      \begin{equation}\label{F(U)}
      \frac{1}{C} m_N(U)\le m_N(F(U))\le C m_N(U).
      \end{equation}
        On the other hand,
      \begin{align*}
      m_N(F(U))&=\int_N 1_{F(U)}d m_N=\int_K(\int_{kH} 1_{F(U)} dm_H)dm'(k)\\
      & \le
       \int_{\pi^{-1}(\bar F(A))\cap K} C m_H(B) dm'(k)=C m_H(B) m_{N/H}(\bar F(A)).
      \end{align*}
      Similarly we obtain $m_N(F(U))\ge \frac{1}{C}m_H(B) m_{N/H}(\bar F(A))$.
    Combining  the above inequalities we get
     $\frac{1}{C^2} m_{N/H}(A)\le m_{N/H}(\bar F(A))\le C^2  m_{N/H}(A)$.

   \end{proof}

   \subsection{Conjugating into  a fiber similarity group.}\label{sec:conjtosim}

Let $1\le i\le s$ be such that  $\dim(N_i/{N_{i-1}})\ge 2$.
Recall  that the quotient   $N_i/{N_{i-1}}$  is a Carnot group.  Let $H_i$ be    the first layer of $\mathfrak n_i/{\mathfrak n_{i-1}} $   and $m_i=\dim(H_i)$.     Notice that  $m_i\ge 2$.

\begin{definition}
A  measurable  fiber conformal structure on the fibration  $\pi_i: N/{N_{i-1}}\ra N/{N_i}$
 is an essentially bounded  measurable map
                    $$\mu: N/{N_{i-1}} \to  X:=SL(m_i)/SO(m_i).$$
\end{definition}
   So a  measurable  fiber conformal structure on the fibration  $\pi_i: N/{N_{i-1}}\ra N/{N_i}$
     can be thought of as  a measurable assignment of inner products (up to scalar multiple)   to horizontal subspaces of the  tangent spaces of the fibers (of the fibration).
For a review of  $SL(n)/SO(n)$ and its Riemannian distance $d_X$ see  Section \ref{conformalstructure}.

For $F:N \to N$ a biLipschitz map we define the pull-back
   $$F^*\mu (pN_{i-1})=(DF_{i, p}(0)\vert_{H_i})^T[\mu(F_{i-1}(pN_{i-1}))], \;\;\;\text{for a.e.}\; pN_{i-1}\in N/{N_{i-1}}.$$
    Notice that the pull-back is well-defined as $F_{i, p}$ depends only on the coset $pN_{i-1}$.

\begin{definition}
We say that $F$ is \emph{conformal with respect to $\mu$} if $F^*\mu (pN_{i-1}) = \mu(pN_{i-1})$ for a.e. $ pN_{i-1}\in N/{N_{i-1}}$.
\end{definition}

     \begin{Prop} \label{existence2}
      Let $\G$ be a countable   uniform quasisimilarity  group
     of $(N,d)$.
     Let $1\le i\le s$ be such that  $\dim(N_i/{N_{i-1}})\ge 2$.
      Then there is a     measurable  fiber  conformal structure $\mu$
        on the fibration  $\pi_i: N/{N_{i-1}}\ra N/{N_i}$
       such that every $\g \in \G$ is conformal with respect to $\mu$.

     \end{Prop}

 \begin{proof}
 There is a   $\Gamma$-invariant
 subset  $U\subset N/{N_{i-1}}$ of full measure  such that  for all $\g \in \G$ and  all
  $pN_{i-1} \in U$ the foliated Pansu derivative $D\gamma_{i, p}(0)$ exists and is a graded automorphism of $N_i/{N_{i-1}}$.
     Let $\mu_0$ be the    $N$-invariant   fiber  conformal   structure
       on the fibration  $\pi_i: N/{N_{i-1}}\ra N/{N_i}$
     that is  associated with an inner product  on $H_i$ (the first layer of $\mathfrak n_i/{\mathfrak n_{i-1}}$).
       For each $pN_{i-1}\in U$, set
        $M_{pN_{i-1}}=\{\g^*\mu_0(pN_{i-1}): \g\in \G$\}.
          Since $\G$ is a uniform quasisimilarity group,  $M_{pN_{i-1}}$ is a bounded subset of $SL(m_i)/SO(m_i)$.
        The assignment $pN_{i-1} \mapsto M_{pN_{i-1}}$ is $\G$-invariant.
          Let   $P(M_{pN_{i-1}})$  be the circumcenter of $M_{pN_{i-1}}$ in  $SL(m_i)/SO(m_i)$  and define $\mu:  U\ra SL(m_i)/SO(m_i)$ by  $\mu(pN_{i-1})=P(M_{pN_{i-1}})$.     Then  $\mu$ is  $\G$-invariant and measurable (see the proof of Proposition \ref{existence}
           for details).
                   \end{proof}

 For the proof of the main result in this section we  need a modified version of Lemma \ref{inmeasure}.

 \begin{lemma}\label{inmeasure2}
   Let $\bar d_{i}$ be a left invariant Carnot-Caratheodory metric on $N_i/{N_{i-1}}$. For a biLipschitz map
   $G: N\ra N$ and any $pN_{i-1}\in N/{N_{i-1}}$,  denote  by $K(G_{i,p}, 0)$ the dilatation of the map
    $G_{i, p}: (N_i/{N_{i-1}}, \bar d_{i})\ra (N_i/{N_{i-1}}, \bar d_{i})$ at $0$.
 Let $\{F^j:  N\ra  N\}$ be a sequence of  $(L,C)$-quasisimilarities
     that  converge uniformly on compact subsets   to  a map $F: N\ra N$.
     Suppose for any compact subset $S \subset N/{N_{i-1}}$ and any $\epsilon>0$ we have
      $$m_{N/{N_{i-1}}}(\{pN_{i-1} : pN_{i-1} \in S,\ K(F^j_{i, p}, 0)\ge 1+\epsilon\})\ra 0$$
      as $j\ra \infty$, where $m_{N/{N_{i-1}}}$ is an invariant measure on the space
        $N/{N_{i-1}}$.       Then  for  every $pN_{i-1} \in N/{N_{i-1}}$,
         $F_{i, p}$ is
         a similarity of $(N_i/{N_{i-1}},   \bar d_{i})$.

 \end{lemma}
 \begin{proof}
 We shall show that for almost every   fiber
  of $\pi_i$, the map  $F_{i, p}$  and a subsequence  of
  $\{F^j_{i, p}\}_j$   satisfy the assumption of Lemma  \ref{inmeasure}
   and so  $F_{i, p}: (N_i/{N_{i-1}}, \bar d_{i})\ra (N_i/{N_{i-1}}, \bar d_{i})$  is conformal. By Theorem 4.1 of   \cite{CO15},   $F_{i, p}$   is a similarity.
     Since the limit of similarity maps  is a  similarity,  the same is true for every  fiber  of $\pi_i$.

 By  Theorem 1.2.12  and its proof  in \cite{CG90},
  the product of a  Haar measure on $N_i/{N_{i-1}}$  and an  invariant measure on $N/{N_i}$
    is an  invariant measure on $N/{N_{i-1}}$,
   in the following sense.   There is a smooth submanifold $Y_i$ of  $N/{N_{i-1}}$   and a measure $m'$ on
 $Y_i$   with the following properties:    (1) the (well-defined) map $P:  Y_i\times N_i/{N_{i-1}}\ra N/{N_{i-1}}$,
   $(yN_{i-1}, xN_{i-1})\mapsto yxN_{i-1}$, is a diffeomorphism and $P_*(m'\times m_{N_i/{N_{i-1}}})$ is an invariant measure on $N/{N_{i-1}}$, where $ m_{N_i/{N_{i-1}}}$   is a Haar measure on
     $N_i/{N_{i-1}}$; (2)     $(\pi_i\circ \iota)_*(m')$ is an
     invariant measure on $N/{N_i}$, where $\iota: Y_i\subset N/{N_{i-1}}$ is the inclusion.

 Let $B_n\subset Y_i$  ($n\ge 1$) be an exhausting sequence of compact subsets of $Y_i$,
  and $A_n\subset   N_i/{N_{i-1}}$  ($n\ge 1$)  be  a sequence of balls
  in  $ N_i/{N_{i-1}}$  centered at the origin $0$   with   radius  going to infinity as $n\ra \infty$.  Set $C_n=P(B_n\times A_n)\subset N/{N_{i-1}}$.
  For integers $n, k,j\ge 1$, let  $B_{n,k,j}=\{pN_{i-1} \in C_n| K(F^j_{i, p}, 0)\ge 1+1/k\}$  and
     $g_{n,k,j}=1_{B_{n,k,j}}$ be the characteristic function of $B_{n,k,j}$.
     Then  ${g_{n,k,j}}\in L^1(C_n)$.  By Fubini,  for a.e. $h\in B_n$, the function
      $g_{n,k,j}^h:  A_n\ra [0,1]$ given by $g_{n,k,j}^h(xN_{i-1})=g_{n,k,j}(P(h,xN_{i-1}))$
        is integrable,   $q_{n,k,j}(h):=\int_{A_n}g_{n,k,j}^h(xN_{i-1}) dm_{N_i/{N_{i-1}}}$ is an integrable function on $B_n$, and  $\int_{B_n}q_{n,k,j}(h)dm'(h)=\int_{C_n}g_{n,k,j}(pN_{i-1})dm_{N/N_{i-1}}=m_{N/N_{i-1}}(B_{n,k,j})$.
     By the assumption, for  fixed $n$ and $k$ we have $m_{N/N_{i-1}}(B_{n,k,j})\ra 0$  as $j\ra \infty$,
    which  implies
       $q_{n,k,j}\ra 0$ in $L^1(B_n)$  as $j\ra\infty$.
 This in turn implies  that  for fixed $n,k$   there is a null set $E_{n,k}\subset B_n$  and a subsequence  $\{q_{n,k,j_l}\}_l$  of  $\{q_{n,k,j}\}_j$
    such that
  $q_{n,k,j_l}(h)\ra 0$ as $l\ra \infty$ for  every  $h\in B_n\backslash E_{n,k}$.
 Set $E=\cup_{k,n} E_{n,k}$. Then $E$ is a null  set in $Y_i$.

 Let $h\in Y_i\backslash E$.   Let $A$ be a compact subset of   $N_i/{N_{i-1}}$ and $\epsilon>0$.   Pick $n,k$ sufficiently large such that $1/k<\epsilon$ and
 $h\in B_n$, $A\subset A_n$.
 By the definition of $E$ we have   $q_{n,k,j_l}(h)\ra 0$ as $l\ra \infty$.
    Due to $DF_{i, p}(xN_{i-1})=DF_{i, px}(0)$,
    we notice that
   \begin{align*}
   0\leftarrow q_{n,k,j_l}(h)
   &=m_{N_i/N_{i-1}}(\{x\in A_n|
    K(F^{j_l}_{i, hx},0)\ge 1+1/k\})\\
    &=m_{N_i/N_{i-1}}(\{x\in A_n|
    K(F^{j_l}_{i, h},x)\ge 1+1/k\})\\
&    \ge m_{N_i/N_{i-1}}(\{x\in A|
    K(F^{j_l}_{i, h},x)\ge 1+\epsilon\})
    \end{align*}
       as $l\ra \infty$  and so  the assumption of Lemma  \ref{inmeasure}  is satisfied by
    $F^{j_l}_{i,h}$, $l\ge 1$, $F_{i, h}$.

\end{proof}

Let $S=N \rtimes_D \R$ be the negatively curved homogeneous space associated with $(N,D)$.  
    As  $\G$ acts cocompactly on the space of distinct pairs of $N$,    every point of $N$ is a radial limit point.  Such an action induces a cobounded quasi-action on $S$ and vice versa.
Recall that   $\mu$ is measurable and hence approximately continuous  at a.e.
     $pN_{i-1}\in N/{N_{i-1}}$.

 \begin{theorem}\label{radialfoliated}  Assume   that  $\dim(N_i/{N_{i-1}})\ge 2$.
  Let $\G$ be a   uniform quasisimilarity group of $N$  and $\mu$ a $\G$-invariant   measurable fiber conformal structure
  on the fibration  $\pi_i: N/{N_{i-1}}\ra N/{N_i}$.
     Suppose there is a point $p\in N$ such that $\mu$ is approximately continuous
      at $pN_{i-1}$ and $p$ is  a radial limit point for $\G$.     Then there exists a
       Carnot-Caratheodory metric $\bar d_i$ on $N_i/{N_{i-1}}$
        and a quasisimilarity $F_0:N\ra N$
            such that $F_0GF_0^{-1}$ consists of elements $F$ where the
             associated $F_{i, p}:    (N_i/{N_{i-1}}, \bar d_i)  \ra (N_i/{N_{i-1}}, \bar d_i)  $  are similarities  for each $p \in N$.
        \end{theorem}

     \begin{proof}
     We assume $p=0$ and use $e^{tD}$ to denote both the dilation on $N$ and the isometry of $S$ that translates along  $\{0\}\times \R$. Similarly elements of $\Gamma$ can be viewed as quasisimilarities of $N$ or quasi-isometries of $S$.
 Since $0$ is a radial limit  point there exists a sequence $\{ h_j\}_{j=1}^\infty$ in $\G$ and a sequence $s_j \to \infty$ such that for any point $x_0 \in S$  we have
    $h_j(x_0)\ra 0$ in  $\bar S=S\cup \partial S$  and    that
  $\{e^{s_jD}\circ h_j(x_0)\}_{j=1}^\infty$   is bounded and hence up to passing to a subsequence we have that
 $F_j=e^{s_jD}\circ h_j$ converges  to a quasisimilarity $F_0: N \to N$
  uniformly on compact subsets.

 We now show that for any $\g \in \G$, the map $F=F_0\g F_0^{-1}$ has the property that  for all $p\in N$ the associated maps $F_{i,p}$ are similarities of $N_i/{N_{i-1}}$ with respect to  the left invariant Carnot-Caratheodory metric   on  $N_i/{N_{i-1}}$
   determined by $\mu(0)$.

      Let  $\g\in \G$ and denote $\tilde\g=F_0\g F_0^{-1}$, $\tilde \g_j=F_j\g F_j^{-1}$.   Let  $\mu_j=(F_j^{-1})^*\mu$. 
                    Since $\mu$ is $\G$-invariant,      $\tilde \g_j$ is conformal with respect to $\mu_j$.
                         Additionally, $\mu_j=(e^{-s_j D})^*\mu$ so for $q\in N$  and $G=e^{-s_j D}$,
                          $$\mu_j(qN_{i-1})=G^*\mu (qN_{i-1})=(DG_{i,q}(0)|_{H_i})^T[\mu(G_{i-1}(qN_{i-1}))]
                          =\mu(G_{i-1}(qN_{i-1}))$$
            since  $DG_{i,q}(0)|_{H_i}$ is a constant   multiple of the
             identity map.

 Let $Z\subset N/{N_{i-1}}$ be a compact subset containing $N_{i-1}\in  N/{N_{i-1}}$.
              Since $\mu$ is approximately continuous at $N_{i-1}\in  N/{N_{i-1}}$  and $s_j\ra \infty$, the equality $\mu_j(qN_{i-1})=\mu(G_{i-1}(qN_{i-1}))$
               implies that    for any $\epsilon>0$    there are subsets $A_j\subset Z$
                with $m_{N/N_{i-1}}(A_j)\ra 0$ as $j\ra \infty$
                  and $d_X(\mu_j(qN_{i-1}), \mu(N_{i-1}))\le \epsilon$ for $qN_{i-1}\in   Z\backslash A_j$.

                   The maps $\tilde\g^{-1}_j$ and $\tilde\g^{-1}$ are all $(L,C)$-quasisimilarities for fixed $L$ and $C$ and so by  Lemma \ref{invariantm}
                   we have $m_{N/N_{i-1}}((\tilde \g_j)_{i-1}^{-1}(A_j))\ra 0$ as $j\ra \infty$.
                            Set $B_j=A_j\cup (\tilde\g_j)_{i-1}^{-1}(A_j)$.
                      Now we have $m_{N/N_{i-1}}(B_j)\ra 0$ as $j\ra \infty$ and
                       $d_X(\mu_j(qN_{i-1}), \mu(N_{i-1}))\le \epsilon$ and $d_X(\mu_j((\tilde\g_j)_{i-1}(qN_{i-1})), \mu(N_{i-1}))\le \epsilon$ for
                        $qN_{i-1}\in Z\backslash B_j$.

     Since  $\tilde\g_j$ is $\mu_j$-conformal, we have
      $\mu_j(qN_{i-1})=(D(\tilde\g_j)_{i,q}(0)|_{H_i})^T [\mu_j((\tilde\g_j)_{i-1}(qN_{i-1}))]$ for a.e. $qN_{i-1}\in N/{N_{i-1}}$.
      Now
   $$    d_X(\mu_j(qN_{i-1}), (D(\tilde\g_j)_{i,q}(0)|_{H_i})^T[\mu(N_{i-1})])
   =    d_X(\mu_j((\tilde\g_j)_{i-1}(qN_{i-1})), \mu(N_{i-1}))\le \epsilon    $$
    for a.e. $qN_{i-1}\in Z\backslash B_j$.
     Combining this with  $d_X(\mu_j(qN_{i-1}), \mu(N_{i-1}))\le \epsilon$ , we   get
      $$
        d_X(\mu(N_{i-1}),  (D(\tilde\g_j)_{i, q}(0)|_{H_i})^T[\mu(N_{i-1})])\le 2\epsilon$$
     for        a.e.  $qN_{i-1}\in Z\backslash B_j$.      By (\ref{K(A)})     this implies   that the assumption of   Lemma \ref{inmeasure2}
 is satisfied   by  $\tilde\g_j$ ($j\ge 1$),  $\tilde\g$
   and so  for any $p\in N$,    $\tilde\g_{i, p}$ is a similarity  of $N_i/{N_{i-1}}$ with respect to  the left invariant Carnot-Caratheodory metric determined by $\mu(N_{i-1})$.

     \end{proof}

     \noindent
     {\bf{Proof of Theorem \ref{foliatedtheorem}}}.
     We observe that we may assume the group $\Gamma$ is countable: Let $\Gamma_0$ be a countable  subgroup  of $\Gamma$ that is dense in the topology of uniform convergence on compact subsets; if  every element of some conjugate of $\Gamma_0$   is  a  $i$-fiber similarity  then the same is true for  $\Gamma$    as the limits of
  $i$-fiber similarity  maps  are $i$-fiber similarity  maps.

     The assumption implies that every point $p\in N$ is a radial limit point of $\Gamma$  and   so Theorem \ref{radialfoliated}
      can be applied.
      List the  elements of $I$ by $i_1<i_2<\cdots< i_k$.    We first conjugate
       $\Gamma$ to get a group  $\Gamma_1$  which is  a $i_1$-fiber  similarity group,
          then conjugate $\Gamma_1$ to get a group $\Gamma_2=F_0\Gamma_1F_0^{-1}$ which is  a $i_2$-fiber  similarity group,  and so on.  We can finish by induction once we  observe  that $\Gamma_2$ is also a $i_1$-fiber  similarity group.    For this we just notice that the conjugating map $F_0$ is the limit  of a sequence of   maps  $F_j$  which are compositions of elements of $\Gamma_1$ and the dilations $e^{tD}$  ($t\in \mathbb R$) of $N$.
Since both $\{e^{tD}| t\in \mathbb R\}$ and $\Gamma_1$ are
     $i_1$-fiber  similarity groups,    each  $F_j$  is a
      $i_1$-fiber  similarity  map. As a consequence, the limiting map $F_0$ is also a
      $i_1$-fiber  similarity  map. Hence $\Gamma_2=F_0\Gamma_1F_0^{-1}$ is     a $i_1$-fiber  similarity group.

 \qed

 \section{A counterexample}\label{example}

In this section we exhibit an example which shows that in general it is impossible to
 conjugate a uniform quasiconformal group
  of $\hat N$ into a conformal group with respect to an arbitrary   pre-specified  Carnot-Caratheodory metric on $N$ when that metric is not maximally symmetric.

Let $N$ be a Carnot  group and $d_{CC}$ a left invariant Carnot-Caratheodory metric on $N$.
   Let
   $IA(N, d_{CC})$   be the group   consisting of   graded automorphisms of $N$ that are also isometries with respect to $d_{CC}$.
     By the main result of \cite{CO15},
  the group  $Conf(N, d_{CC})$
    of conformal maps of $(N, d_{CC})$
    is  generated by left translations, standard  Carnot dilations and $IA(N, d_{CC})$.

\begin{lemma}\label{twometrics}
     Let $N$ be a Carnot group and $d_1, d_2$ two left invariant  Carnot-Caratheodory metrics on $N$.
      Suppose there is a quasiconformal map $f: (N, d_1)\ra (N, d_2)$ such that
         $f \cdot Conf(N, d_1)\cdot  f^{-1}\subset Conf(N, d_2)$, then
           there is a graded automorphism $h$ of $N$ such that
            $h\cdot IA(N, d_1) \cdot h^{-1}\subset IA(N, d_2)$.

\end{lemma}

\begin{proof}  Denote by $\phi:  Conf(N, d_1)\ra  Conf(N, d_2)$ the injective homomorphism given by:
       $\phi(g_1)=f g_1 f^{-1}$. Let $p\in N$ be a point  such that    $Df(p)$  exists and  is a graded automorphism.
   After pre-composing and post-composing  with  left translations, we may assume $p=f(p)=0$ is the origin.
     Set $h=Df(0)$.

Let $g\in IA(N, d_1)$.  Then $g(0)=0$. It follows that $\phi(g)(0)=0$  and so  $\phi(g)$ is the composition of a standard  Carnot dilation   and
      an element of $IA(N, d_2)$.
     Since $\{g^i: i\in \mathbb N\}$ has compact closure, the same is true for  $\{\phi(g)^i: i\in \mathbb N\}$.    This follows from the fact that conjugation by $f$ is continuous. Alternatively one can argue using the fact that $f$ is quasisymmetric.
      This implies that $\phi(g)\in IA(N, d_2)$.    Being graded automorphisms, both $g$ and    $\phi(g)$ commute with  Carnot dilations.
         Now   for any $t>0$,
 $$\delta_t f\delta^{-1}_t       g  (\delta_t f\delta^{-1}_t )^{-1}=\delta_t f\delta^{-1}_t       g  \delta_t f^{-1}  \delta^{-1}_t
 =\delta_t f     g   f^{-1}  \delta^{-1}_t=\delta_t \phi(g)  \delta^{-1}_t=\phi(g).$$
   Since $\delta_t f \delta^{-1}_t$ converges uniformly on compact subsets to $h$ as $t\ra +\infty$, we have
      $h g  h^{-1}=\phi(g)\in IA(N, d_2)$.    Since this is true for every $g\in IA(N, d_1)$,  the lemma follows.

\end{proof}

 We note that
  for any two left invariant Carnot-Caratheodory metrics   $d_1, d_2$   on $N$,
      the group $Conf(N, d_1)$   is a uniform quasiconformal  group of  $(N, d_2)$.    Furthermore,
      the action of $Conf(N, d_1)$    on $ (N, d_2)$ satisfies the assumption of Theorem   \ref{tukia}.
         If   every uniform quasiconformal group of $\hat N$ satisfying the assumptions of Theorem \ref{tukia}
             can be conjugated into the conformal group  with respect to $d_2$,   then
              there is a quasiconformal map $f: \hat N\rightarrow \hat N$  such that for any
               $g\in Conf(N, d_1)$, the map $f g f^{-1}:  (N\backslash\{fg^{-1}f^{-1}(\infty)\}, d_2)\ra      (N\backslash\{fgf^{-1}(\infty)\}, d_2) $
                is conformal.
             Suppose  $N$ has the property that every quasiconformal map of $\hat N=N\cup \{\infty\}$ fixes $\infty$.  Then we have
             $(f|_N)\text{Conf}(N, d_1)(f|_N)^{-1}\subset \text{Conf}(N, d_2)$ and by Lemma \ref{twometrics} we conclude that
    $IA(N, d_1)$ injects into $IA(N, d_2)$.
         Next we exhibit an example where  such an injection does not exist.

  Let  $N=H\times H$ be the direct product of the first Heisenberg group with itself.
     Theorem 1.1 in   \cite{KMX20}  implies that    every quasiconformal map of $\hat N=N\cup \{\infty\}$ fixes $\infty$.
    The  Lie algebra of $N$ can be written as
 $\mathfrak n=\mathfrak h\oplus \mathfrak h=V_1\oplus V_2$  with  first layer $V_1=\R^2\oplus \R^2$, where the first $\R^2$ is the first layer of the first $\mathfrak h$ and the second $\R^2$ is the first layer of the   second
 $\mathfrak h$.
Let $e_1, e_2$ denote the standard basis in the first $\R^2$, and $\tilde e_1, \tilde e_2$ denote the standard basis in the second $\R^2$.
 We consider two different left invariant Carnot-Caratheodory metrics   $d_1$, $d_2$  on $N$.
The metric $d_1$ is determined by   the inner product    $<, >_1$   on  $V_1$   that has $e_1$, $e_2$, $\tilde e_1$, $\tilde e_2$ as an orthonormal basis.
The metric $d_2$ is determined by   the inner product  $<, >_2$
on  $V_1$   that has $e_1$, $e_2$, $\tilde e_1$, $\frac{\sqrt 2}{2}(\tilde e_2-e_1)$ as an orthonormal basis.

\begin{lemma}\label{isograded}
(1)
$IA(N, d_1)$ is isomorphic to
 $(O(2)\oplus O(2))\rtimes \Z_2$,  where  the generator of $\Z_2$     acts on  $O(2)\oplus O(2)$    by $(M_1, M_2)\ra (M_2, M_1)$;\newline
(2)
  $IA(N, d_2)$ is isomorphic to
   $(\Z_2\oplus \Z_2\oplus \Z_2)\rtimes \Z_2$,
where  the generator of $\Z_2$     acts on    $\Z_2\oplus \Z_2\oplus \Z_2$  by
  $(a,b,c)\ra (b,a,   c)$.
\end{lemma}

\begin{proof}  For any element $X$ of a Lie algebra $\mathfrak n$, let $r(X)$ be the rank of the linear map
 $$ad(X): \mathfrak n\ra \mathfrak n.$$
   Then  $r(X)=r(A(X))$ for any  $x\in \mathfrak n$ and  any   isomorphism $A$ of Lie algebras.
 Notice that     $r(x,y)\ge 1$ for every nonzero element
 $(x,y)\in   V_1=\R^2\oplus \R^2$,  and furthermore
 $r(x,y)= 1$   if and only if one of the following happens:\newline
  (a) $x=0$ and $y\not=0$; \newline
(b) $y=0$ and $x\not=0$. \newline
  It follows that for any graded isomorphism $A: \mathfrak  n\ra \mathfrak n$, we have
one of the following:\newline
(i) $A(\R^2\oplus \{0\})=\R^2\oplus \{0\}$  and
$A(\{0\}\oplus \R^2)=\{0\}\oplus \R^2$;\newline
(ii)
$A(\R^2\oplus \{0\})=\{0\}\oplus \R^2$  and
$A(\{0\}\oplus \R^2)=\R^2\oplus \{0\}$.

(1) follows easily from the above paragraph.   \newline
 (2)  Now   let  $A\in IA(N, d_2)$.
First assume  $A$ satisfies (i).
  Note the orthogonal complement of $\R^2\oplus \{0\}$ in
  $(V_1, <, >_2)$    is  $E_1:=\R\tilde e_1\oplus \R (\frac{\sqrt 2}{2}\cdot (\tilde e_2-e_1))$,
  and   similarly the orthogonal complement of  $\{0\}\oplus \R^2$
  is
$E_2:=\R e_2\oplus \R(e_1-\frac{1}{3} \tilde e_2)$.
 Since $A$ preserves orthogonal complement, we have  $A(E_1)=E_1$  and $A(E_2)=E_2$.
  Note  $E_1\cap (\{0\}\oplus \R^2)=\R \tilde e_1$  and  $E_2\cap  (\R^2\oplus \{0\})=\R e_2$.
   It follows that $A(\R \tilde e_1)=\R \tilde e_1$  and $A(\R e_2)=\R  e_2$.
Then $A$ also preserves the orthogonal complement of $\R \tilde e_1$  in
 $\{0\}\oplus \R^2$  and that of $\R  e_2$   in  $\R^2\oplus \{0\}$.  That is, we have
 $A(\R \tilde e_2)=\R \tilde e_2$  and $A(\R e_1)=\R  e_1$.
  From this it is easy to see that there are only  8  such isometric graded isomorphisms. They are given by
 $A(\tilde e_1)=\epsilon_1 \tilde e_1$,  $A(e_2)=\epsilon_2 e_2$,  $A(e_1)=\epsilon_3 e_1$
 and  $A(\tilde e_2)=\epsilon_3 \tilde e_2$,   where $\epsilon_1, \epsilon_2, \epsilon_3\in \{1, -1\}$.
  They form a group $F$ isomorphic to $\Z_2\oplus \Z_2\oplus \Z_2$.

 A similar argument shows that there are also 8   isometric graded isomorphisms   with respect to $d_2$
  that satisfy (ii).
  They are given by  $A(\tilde e_1)=\epsilon_1 e_2$,  $A(e_2)=\epsilon_2  \tilde e_1$,
 $A(e_1)=\frac{1}{\sqrt 3}\epsilon_3 \tilde e_2$  and $A(\tilde e_2)={\sqrt 3}\epsilon_3  e_1$,
where $\epsilon_1, \epsilon_2, \epsilon_3\in \{1, -1\}$.
  If we denote by $A_0$ the isometric   graded isomorphism  corresponding  to $\epsilon_1=\epsilon_2=\epsilon_3=1$, then
  these 8 isomorphisms are simply $A_0\cdot F$.   Now it is easy to see that
    (2) holds.

\end{proof}

\end{document}